\documentclass[reqno, 12pt]{amsart}

\usepackage[colorlinks, linkcolor=blue, citecolor=red, urlcolor=cyan,pagebackref]{hyperref}

\allowdisplaybreaks

\usepackage{amsthm,amssymb,amscd,amsmath}
\usepackage{mathrsfs}
\usepackage{bbm}
\usepackage{graphicx}
\usepackage{epic,eepic}
\usepackage{tikz-cd}
\usepackage{stmaryrd}
\usepackage{mathtools}
\usepackage{stmaryrd}
\usepackage{diagbox}

\usepackage{ascmac}
\usepackage[normalem]{ulem}

\usepackage{accents}

\usepackage{here}
\usepackage{bm}
\usepackage[all]{xy}

\usepackage{xcolor} 
\definecolor{mygreen}{rgb}{0,0.7,0.3}
\definecolor{myblue}{rgb}{0,0.50,1.20}
\definecolor{myorange}{rgb}{1,0.5,0.1}

\definecolor{fillred}{rgb}{1,0.9,0.9}
\definecolor{fillgreen}{rgb}{0.9,1,0.9}

 \definecolor{refkey}{rgb}{0,0.7,0.3}
 \definecolor{labelkey}{rgb}{1,0,0}

\usepackage{tikz}
\usetikzlibrary{calc,decorations.markings}
\usetikzlibrary{arrows.meta}
\usetikzlibrary{positioning}
\usetikzlibrary{cd, intersections, calc, decorations.pathmorphing, arrows, decorations.pathreplacing}

\numberwithin{equation}{section}

\usepackage{version}

\usepackage{thmtools} 
\usepackage{cleveref}

\declaretheorem[name=Theorem,numberwithin=section]{thm}
\declaretheorem[sharenumber=thm,name=Corollary]{cor}
\declaretheorem[sharenumber=thm,name=Lemma]{lem}

\declaretheorem[sharenumber=thm,name=Proposition]{prop}

\declaretheorem[name=Theorem]{introthm}
\declaretheorem[sharenumber=introthm,name=Conjecture]{introconj}

\theoremstyle{definition}
\declaretheorem[sharenumber=thm,name=Definition]{dfn}

\declaretheorem[sharenumber=thm,name=Notation]{conv}
\declaretheorem[sharenumber=thm,name=Example]{ex}
\declaretheorem[sharenumber=thm,name=Remark]{rem}

\crefname{thm}{Theorem}{Theorems}
\crefname{cor}{Corollary}{Corollaries}
\crefname{lem}{Lemma}{Lemmas}
\crefname{lemdef}{Lemma-Definition}{Lemma-Definitions}
\crefname{prop}{Proposition}{Propositions}
\crefname{dfn}{Definition}{Definitions}
\crefname{defi}{Definition}{Definitions}
\crefname{ex}{Example}{Examples}
\crefname{claim}{Claim}{Claims}
\crefname{conj}{Conjecture}{Conjectures}
\crefname{conv}{Notation}{Notations}
\crefname{rem}{Remark}{Remarks}
\crefname{rmk}{Remark}{Remarks}
\crefname{prob}{Problem}{Problems}
\crefname{figure}{Figure}{Figures}
\crefname{table}{Table}{Tables}
\crefname{introthm}{Theorem}{Theorems}
\crefname{introconj}{Conjecture}{Conjectures}
\crefname{section}{Section}{Sections}
\crefname{subsection}{Section}{Sections}

\newcommand{\bZ}{\mathbb{Z}}
\newcommand{\bQ}{\mathbb{Q}}
\newcommand{\bR}{\mathbb{R}}
\newcommand{\bC}{\mathbb{C}}

\newcommand{\bP}{\mathbb{P}}
\newcommand{\bG}{\mathbb{G}}
\newcommand{\bA}{\mathbb{A}}

\newcommand{\bi}{{\mathbf{i}}}

\newcommand{\by}{{\mathbf{y}}}

\newcommand{\cF}{\mathcal{F}}

\newcommand{\cI}{\mathcal{I}}

\newcommand{\cN}{\mathcal{N}}
\newcommand{\cO}{\mathcal{O}}

\newcommand{\X}{\mathcal{X}}

\newcommand{\fS}{\mathfrak{S}}

\newcommand{\ve}{\varepsilon}

\newcommand{\Hom}{\mathrm{Hom}}
\newcommand{\sgn}{\mathrm{sgn}}

\newcommand{\mtx}[1]{\begin{pmatrix} #1 \end{pmatrix}}
\newcommand{\Ind}{\mathcal{I}}
\newcommand{\Fix}{\mathrm{Fix}}

\DeclareMathOperator{\Spec}{\mathrm{Spec}}

\DeclareMathOperator{\End}{\mathrm{End}}
\DeclareMathOperator{\Tr}{\mathrm{Tr}}

\allowdisplaybreaks

\makeatletter
\newcommand{\oset}[3][0ex]{%
  \mathrel{\mathop{#3}\limits^{
    \vbox to#1{\kern-2\ex@
    \hbox{$\scriptstyle#2$}\vss}}}}
\makeatother

\makeatletter
\newcommand{\osetnear}[3][0ex]{%
  \mathrel{\mathop{#3}\limits^{
    \vbox to#1{\kern-.3\ex@
    \hbox{$\scriptstyle#2$}\vss}}}}
\makeatother

\tikzset{
  qarrow/.style={->,shorten >=2pt,shorten <=2pt,>=latex
      },
}
\tikzset{
  head/.style={shorten >=4pt
      },
}
\tikzset{
  tail/.style={shorten <=4pt
      },
}
\newcommand{\qarrow}[2]{\draw[qarrow](#1) -- (#2)}

\tikzset{
  mid arrow/.style={postaction={decorate,decoration={
        markings,
        mark=at position .5 with {\arrow[#1]{stealth}}
      }}},
}

\tikzset{->-/.style 2 args={
	postaction={decorate},
	decoration={markings, mark=at position #1 with {\arrow[#2]{>}}} 
    },
    ->-/.default={0.5}{}
}
\tikzset{-<-/.style 2 args={
	postaction={decorate},
	decoration={markings, mark=at position #1 with {\arrow[#2]{<}}} 
    },
    -<-/.default={0.5}{}
}

\tikzset{
  qqarrow/.style={->,->-={0.87}{},shorten >=2pt,shorten <=2pt,>=latex
      },
}
\newcommand{\qqarrow}[2]{\draw[qqarrow](#1) -- (#2)}

\pgfdeclarelayer{bg}    
\pgfsetlayers{bg,main}  

\title{Index for a mutation loop}

\author[Tsukasa Ishibashi]{Tsukasa Ishibashi}
\address{Tsukasa Ishibashi, Mathematical Institute, Tohoku University, 
6-3 Aoba, Aramaki, Aoba-ku, Sendai, Miyagi 980-8578, Japan.}
\email{tsukasa.ishibashi.a6@tohoku.ac.jp}

\author[Yuji Terashima]{Yuji Terashima}
\address{Yuji Terashima, Mathematical Institute, Tohoku University, 
6-3 Aoba, Aramaki, Aoba-ku, Sendai, Miyagi 980-8578, Japan.}
\email{}

\date{\today}

\begin{document}

\begin{abstract}
Motivated by the 3d-3d correspondence, we introduce an index $\mathcal{I}_g^\phi$ for a mutation loop $\phi$ of a quiver and an integer $g \geq 0$. We prove its integrality $\mathcal{I}_g^\phi \in \bZ$ under certain conditions, which extends the integrality conjecture of Gang--Kim--Yoon \cite{GKY} for $3$-manifolds to the broader framework of mutation loops in cluster algebras. Our index also has an interpretation as the genus-$g$ partition function of a $(1+1)$-dimensional TQFT. We provide explicit computations for mutation loops of period one. In particular, in the case of the generalized Kronecker quivers, we give an explicit formula expressing the generating function of the indices for all genera $g$ as a ratio of polynomials. Altogether, we provide a new insight to the interplay between the geometry of $3$-manifolds and cluster algebra.
\end{abstract}

\maketitle

\setcounter{tocdepth}{1} 
\tableofcontents

\section{Introduction}\label{sec:intro}



In the rich and fruitful interplay between the physics and mathematics of quantum fields, quantum field theory continues to provide new perspectives and conjectures that often appear unexpected from a mathematical point of view. In particular, recent developments have revealed a number of intriguing relationships between quantum quantities arising in quantum field theory and topological invariants in mathematics. Recently, Gang, Kim, and Yoon \cite{GKY}, based on the 3d–3d correspondence \cite{TY1,TY2,DG,DGG,CCV}, proposed that, for a three-dimensional gauge theory on $\Sigma_g \times S^1$
labeled by a $3$-manifold $M$, the index $I_g$ can be identified with a sum of $(g-1)$-st powers of Reidemeister torsions associated with $M$. See also \cite{GKP,BGP}. This correspondence provides a bridge between the geometric information of a $3$-manifold and the quantum information of the corresponding gauge theory, and offers an important clue for uncovering new mathematical relationships between the two. 

Based on this proposal, they formulated a mathematical conjecture asserting that this sum of $(g-1)$-st powers of Reidemeister torsions should be an integer as a consequence of the integrality of the index \cite[Conjecture 3.1]{GKY}. Moreover, for $g=0$, they conjecture that this sum, namely the inverse sum of Reidemeister torsions, should be zero \cite[Conjecture 1.1]{GKY}. 
For mathematical proofs of the integrality conjecture for twist knots, see \cite{TTY26}, and for torus knots, see \cite{TYa26,MT26}. 

On the other hand, the 3d–3d correspondence was extended in \cite{TY} to the setting of three-dimensional gauge theories labeled by mutation loops of quivers rather than by $3$-manifolds. This extension makes it possible to transfer structures arising from $3$-manifolds into the framework of cluster theory \cite{FZ-CA1,FG09}. \cref{tab:cluster_topology} gives a first dictionary between the cluster theory and $2$ \& $3$-manifolds. 

\begin{table}[ht]
    \centering
    \begin{tabular}{l|l}
    Cluster theory  & $2$ \& $3$-manifolds  \\ \hline
    Cluster Poisson variety $\X$  & framed character variety $\X_{G,\Sigma}$ of a surface $\Sigma$ \cite{FG06}\\
    Mutation loop $\phi$ & Mapping class $\phi$ on $\Sigma$ \\
    Fixed point locus $\Fix_\phi$ & framed character variety $\X_{G,M_\phi}$ of the mapping torus $M_\phi$ \\ [2mm]
    \end{tabular}
    \caption{Dictionary between the cluster theory and $2$ \& $3$-manifolds}
    \label{tab:cluster_topology}
\end{table}

\bigskip
\subsection{Indices of a mutation loop}
In this paper, combining these two perspectives, we introduce a definition of the \emph{genus-$g$ index} $\Ind_g^\phi$ associated with a non-degenerate mutation loop $\phi$ for $g \geq 0$ (\cref{def:index}):
\begin{align*}
    \Ind_g^\phi:= \sum_{x \in \Fix_\phi(\overline{\bQ})} \det(1-(d\phi)_x)^{g-1} \in \overline{\bQ}.
\end{align*}
Here, $x$ runs over the $\overline{\bQ}$-fixed points of $\phi$ in the cluster Poisson torus. The non-degeneracy means $\det(1-(d\phi)_x) \neq 0$ for any $x$. In particular, the fixed point locus is $0$-dimensional. In view of \cite{KT}, each term $\det(1-(d\phi)_x)$ is a cluster-theoretic analogue of the non-acyclic adjoint Reidemeister torsion at $x$ (\cref{rem:torsion}). 
We also introduce a one-parameter refinement $\Ind_{g,t}^\phi \in \overline{\bQ}(t)$ (\cref{def:refined_index}), which recovers the index by $\Ind_{g,1}^\phi=\Ind_g^\phi$. See \cref{ex:A2_index,ex:A2_refined} for examples of type $A_2$. 

\subsection{Integrality for $g \geq 1$}
We then investigate the integrality of these (refined) indices. Our main result is the following:
\begin{introthm}[\cref{thm:finite_order,prop:shape_integral}]
Let $\phi$ be a non-degenerate mutation loop.
\begin{enumerate}
    \item If $\phi$ has finite order, then $\Ind_{g,t} \in \bZ[t]$ for all $g \geq 1$.
    \item Under a certain assumption on integrality of shape parameters, we have $\Ind_{g,t} \in \bZ[t]$ for all $g \geq 1$. 
\end{enumerate}
\end{introthm}
See \cref{prop:shape_integral} for a precise condition on the second statement. Proofs of both statements are based on the interpretation of $\Ind_{g}^\phi$ as a trace in the Artinian ring $R_\phi=\cO(\Fix_\phi)$. The second statement further relies on the Mizuno's Jacobian formula \cite{Mizuno:Jac}. 

\begin{introconj}\label{conj:integrality}
Let $\phi$ be a non-degenerate mutation loop of an indecomposable seed. Assume $\phi$ is either finite-order or primitive (\cref{def:primitive}).
Then $\Ind_g^\phi \in \bZ$ for any $g \geq 1$.    
\end{introconj}

This gives a cluster-algebraic analogue of the integrality conjecture of Gang--Kim--Yoon for $3$-manifolds to the broader framework of mutation loops in cluster algebras. As we will see in \cref{ex:refined_non_integral}, the integrality of the refined index turns out not to be true in general.

\subsection{Frobenius structure and the vanishing for $g=0$}
In \cref{subsec:Frob_T}, we also provide a $(1+1)$-dimensional interpretation of our index. The fixed point algebra $R_\phi$ admits a natural Frobenius structure such that $\Ind_g^\phi$ can be interpreted as the genus-$g$ partition function $Z(\Sigma_g)$ of the associated $(1+1)$-dimensional TQFT (\cref{thm:TQFT}). This description suggests an intimate connection to the Vafa--Intriligator formula \cite{SiebertTian} on quantum cohomology. A presentation of $R_\phi$ as an Artinian complete intersection leads to another Frobenius structure on $R_\phi$ (\cref{subsec:residue}). Following the work of Siebert--Tian \cite{SiebertTian}, we then obtain a sufficient condition for the vanishing of the genus-$0$ index $\Ind_0^\phi$ by a degree argument. This is based on the global residue theorem: a connection between the vanishing conjecture and the global residue theorem was also observed by
Gang--Kim--Yoon \cite{GKY}. 

\subsection{Examples}
We exhibit our results for the mutation loops of period one, which have been classified by Fordy–Marsh \cite{FM11}, including the case of generalized Kronecker quivers. The result is as follows:
\begin{introthm}[\cref{thm:Fordy_Marsh,thm:FM_ind0}]
Let $\phi$ be a mutation loop of period one associated with a palindromic sequence $\mathbf{a}=(a_2,\dots,a_N)$, and let $A^\pm:=2-\sum_{j=2}^N[\pm a_j]_+$. Assume $(A^+,A^-)\neq (0,0)$.
\begin{enumerate}
    \item We have $\Ind_g^\phi \in \bZ$ for all $g \geq 1$. If moreover $A^+\neq A^-$, then $\Ind_{g,t}^\phi \in \bZ[t]$ for $g \geq 1$.
    \item $\Ind_0^\phi=0$ if and only if $A^+A^- <0$ or $(A^+,A^-) \in \{(\pm 1,0), (0,\pm 1)\}$.
\end{enumerate}
\end{introthm}
Several remarks are in order. 
(1) When $A^+=A^-$, we have examples for which $\Ind_g^\phi \in \bZ$ for all $g \geq 1$ but $\Ind_{g,t}^\phi \notin \bZ[t]$ (\cref{ex:refined_non_integral}). For the case of generalized Kronecker quivers, explicit values of $\Ind_g^\phi$ and $\Ind_{g,t}^\phi$ are shown in \cref{tab:Kronecker_index} and \cref{tab:Kronecker_refined_index}, respectively. 
(2) See \cref{thm:FM_ind0} for precise values of $\Ind_0^\phi$ including the non-zero cases. \cref{tab:index-g0} shows the explicit values for small $(A^+,A^-)$. The non-vanishing is possibly related to ``missing'' fixed points outside the initial cluster torus (cf. \cref{rem:GHK}). We do not pursue this direction in this paper.

In the special case of generalized Kronecker quivers, we also give an explicit formula expressing the generating function of the indices $\Ind_g^\phi$ for all genera $g$ as a ratio of polynomials (\cref{thm:generating}). This generating-function description allows us not only to compute the index for each genus, but also to describe all genera simultaneously in a unified manner.

\subsection{Further perspectives}
The definition of the index introduced in this paper is based on a cluster-theoretic
analogue of Reidemeister torsion for $3$-manifolds. We expect that this analogue will play an important role as a new framework for unifying $3$-manifolds and mutation loops in cluster algebras.
In particular, it may provide a cluster-theoretic
counterpart of projects associating modular tensor categories to $3$-manifolds \cite{CGK,CQW}, the knot–quiver correspondence \cite{KRSS1,KRSS2,EKL}, and the knot–gauge correspondence \cite{MTT}. More broadly, these connections may reveal new relationships between cluster algebra
and gauge theory that are not visible from the existing frameworks.

\subsection*{Organization of the paper}
In \cref{sec:cluster}, we recall the basics on mutation loops, their description by mutation networks \cite{TY,Mizuno:Jac}, and their refinement by Mizuno \cite{Mizuno:thesis}. These provide a useful tool to investigate the fixed point locus. 
In \cref{sec:index}, we introduce the index of a mutation loop, its refined version, and the generating function with the simplest example of type $A_2$. The integrality of the (refined) index for $g \geq 1$ is discussed in \cref{sec:integrality}. \cref{sec:Frobenius} provides a $(1+1)$-dimensional interpretation, and a vanishing condition for the genus-$0$ index. Finally, in \cref{sec:example}, we give explicit computations for the mutation loops of period one.

\subsection*{Acknowledgements}
  The authors would like to express our gratitude to Y.~Mizuno, R.~Tange, and Y.~Yamaguchi for helpful discussions. 
  T. I. is supported by JSPS KAKENHI Grant Number JP24K16914.
  Y.T. is supported by JSPS KAKENHI Grant Numbers 23K22388, 25K06969, 26K22256. 

\subsection*{AI disclosure}
The authors made limited use of ChatGPT 6-Astra (OpenAI) for preliminary mathematical discussions, consistency checks, assistance with Mathematica coding, grammar correction and search for (counter-)examples. No counterexample to \cref{conj:integrality} has been found in this use.
All mathematical arguments have been independently reconstructed and finalized by the authors. 
The manuscript was written entirely by the authors, and all the references were independently verified. 
The authors take full responsibility for the content of this paper.

\section{Mutation loops}\label{sec:cluster}

\subsection{Seeds and mutations}
We fix the following data:
\begin{itemize}
	\item A finite set $I$ of indices.
	\item A field $\cF$ isomorphic to the field $\bQ(z_i \mid i \in I)$ of rational functions in $|I|$ variables with coefficients in $\bQ$, called the \emph{ambient field}. 
\end{itemize}

\begin{dfn}
A \emph{(labeled, skew-symmetric) seed} in $\cF$ is a pair $\bi=(B,\mathbf{y})$, where
\begin{itemize}
    \item $B=(b_{ij})_{i,j \in I}$ is an $I \times I$ skew-symmetric integral matrix.
    \item $\mathbf{y} = (y_i)_{i \in I}$ is a tuple of elements of $\cF$ such that $\cF \cong \bQ(y_i \mid i \in I)$. 
\end{itemize}
We call the matrix $B$ the \emph{exchange matrix} of $\bi$. Each element $y_i$ is called a \emph{cluster Poisson variable}.
\end{dfn}
The exchange matrix $B$ can be encoded in a quiver $Q$ with vertex set $I$ such that 
\begin{align*}
    b_{ij}=\#\{\text{arrows $i \to j$}\}  - \#\{\text{arrows $j \to i$}\}.
\end{align*}

We then define \emph{mutations} of seeds. 
We will use the notation $[x]_+:=\max\{x,0\}$, and let $\sgn (x) \in \{-1,0,1\}$ denote the sign for any $x \in \bR$. 

\begin{dfn}
Let $\bi$ be a seed in $\cF$. 
For each $k \in I$, the \emph{seed mutation directed to $k$} creates a new seed $\bi'=\mu_k(\bi)=(B',\mathbf{y}')$ in $\cF$ by the following rule.
\begin{align}
	b'_{ij} &:= \begin{cases}
		-b_{ij} & \mbox{if $k \in \{i,j\}$}, \\
		b_{ij} + [b_{ik}]_+ [b_{kj}]_+ - [-b_{ik}]_+ [-b_{kj}]_+ & \mbox{otherwise}.
	\end{cases} \label{eq:mutation_matrix} \\
	y'_i &:= \begin{cases}
		y_k^{-1} & \mbox{if $i=k$}, \\
		y_i(1+y_k^{-\sgn (b_{ki})})^{-b_{ki}} & \mbox{if $i \neq k$}.
	\end{cases} \label{eq:mutation_X}
\end{align}
We call the transformation \eqref{eq:mutation_matrix} the \emph{matrix mutation}, the transformation \eqref{eq:mutation_X} the \emph{cluster Poisson transformation}.
\end{dfn}
A seed $\bi=(B,\mathbf{y})$ is said to be \emph{indecomposable} if $B$ cannot be decomposed into a non-trivial direct sum. It is equivalent to the condition that the associated quiver $Q$ is connected. This condition is mutation-invariant.




\subsection{Mutation loops}
One can consider any successive application of mutations $\mu_{k_n}\dots \mu_{k_2}\mu_{k_1}(\bi)$ for $k_1,\dots,k_n \in I$. 
For $\sigma \in \fS_I$, we define a new seed $\bi'=\sigma(\bi)=(B',\mathbf{y}')$ by the rule
\begin{align}\label{eq:permutations}
	b'_{ij} := b_{\sigma^{-1}(i),\sigma^{-1}(j)}, \quad y'_j := y_{\sigma^{-1}(j)}.
\end{align}
Note the relation $\sigma \mu_k \sigma^{-1} = \mu_{\sigma(k)}$ for $k \in I$, $\sigma \in \fS_I$.

\begin{dfn}
Let $\bi$ be a seed in $\cF$. A sequence of seeds
\begin{align}\label{eq:mutation_seq}
    \phi: \bi=\bi^{(0)} \to \bi^{(1)} \to \dots \to \bi^{(L)},
\end{align}
where each step $\bi^{(\nu)} \to \bi^{(\nu+1)}$ is either a mutation or a permutation of labels, is called a \emph{mutation sequence} with initial seed $\bi$. Write $\bi^{(\nu)}=(B^{(\nu)},\mathbf{y}^{(\nu)})$ for $\nu=0,\dots,L$. 
The mutation sequence $\phi$ is called a \emph{mutation loop} if $B^{(L)}=B^{(0)}$. We identify two mutation loops if the terminal seed $\bi^{(L)}$ is the same.
We may also write $\phi$ as a word of symbols $\mu_k$ ($k \in I$) and $\sigma \in \fS_I$. 
\end{dfn}
Let $\bG_m$ be the multiplicative group scheme such that $\bG_m(K)=K^\ast$ for any field $K$. The cluster Poisson transformation \eqref{eq:mutation_X} can be viewed as a birational map 
\begin{align*}
    \mu_k: (\bG_m)^I \dashrightarrow (\bG_m)^I,
\end{align*}
where the variables $(y_i)_{i \in I}$ and $(y'_i)_{i \in I}$ are regarded as coordinate functions on $(\bG_m)^I$ on the source and the target, respectively. 
A mutation loop $\phi$ induces a birational map $\phi: (\bG_m)^I \dashrightarrow (\bG_m)^I$, still denoted by the same symbol, as a composite of these birational maps corresponding to mutations and coordinate permutations. 
In other words, its pull-back action
\begin{align*}
    \phi^\ast: \bQ(y_i \mid i \in I) \to \bQ(y_i \mid i \in I)
\end{align*}
on coordinates is given by the composite of the cluster Poisson transformations \eqref{eq:mutation_X} and the permutations \eqref{eq:permutations} in the reversed order. 

\begin{dfn}[Fixed point locus]
Let $U_\phi \subset (\bG_m)^I$ be the maximal open subscheme where the $\phi$-action is defined. Then we define the \emph{fixed point locus} 
\begin{align*}
    \Fix_\phi \subset U_\phi
\end{align*}
as the equalizer of $\phi|_{U_\phi}: U_\phi \to (\bG_m)^I$ and the inclusion $U_\phi \to (\bG_m)^I$. 
We denote by $R_\phi:=\cO(\Fix_\phi)$ the $\bQ$-algebra of regular functions on $\Fix_\phi$.
\end{dfn}

\begin{ex}[Type $A_2$]\label{ex:A2_loop}
Consider the seed $\bi=(\begin{psmallmatrix}0 & -1 \\ 1 & 0\end{psmallmatrix},(y_1,y_2))$ of type $A_2$, and the mutation loop $\phi: \bi=\bi^{(0)} \xrightarrow{\mu_1} \bi^{(1)} \xrightarrow{(1\ 2)}\bi^{(2)}$. We have $\phi=(1\ 2)\mu_1$ in the word notation.
The induced birational map $\phi: (\bG_m)^2 \dashrightarrow (\bG_m)^2$ is computed on coordinates as
\begin{align}\label{eq:A2_rational}
    \phi^\ast: \mtx{y_1 \\ y_2} \mapsto \mtx{y_2 \\ y_1} \mapsto \mtx{y_2(1+y_1) \\ y_1^{-1}}.
\end{align}
The fixed point locus is defined by the equations $y_1=y_2(1+y_1)$, $y_2=y_1^{-1}$. Observe that it does not intersect with the line $y_1=-1$ on which $\mu_1$ is not defined.
Therefore
\begin{align}\label{eq:A2_ring}
    R_\phi \cong \bQ[y]/(y^2-y-1)
\end{align}
with $y:=y_1$.
\end{ex}

The following formula is useful: 

\begin{prop}[Jacobian formula {\cite[Proposition 4.1]{Mizuno:Jac}}]\label{prop:Jacobian}
Let $\bi'=\mu_k(\bi)$ be a mutation directed to $k \in I$, and $p \in \bG_m(\overline{\bQ})^I$ a point 
such that $y_k(p)\neq -1$. 
Then, the Jacobian matrix $J(\mu_k)_{p}:=\displaystyle\bigg(\frac{\partial y'_i}{\partial y_j}(p)\bigg)_{i,j \in I}$ is given by
\begin{align*}
    J(\mu_k)_{p}=\by'(\mu_k(p))\cdot \big(z_-(p) F_{k,+}(B)^\top + z_+(p) F_{k,-}(B)^\top \big) \cdot\by^{-1}(p),
\end{align*}
where
\begin{align*}
    &z_+:=\frac{y_k}{1+y_k}, \quad z_-:=\frac{1}{1+y_k}, \\
    &\by(p) :=\mathrm{diag}(y_i(p) \mid i \in I), \quad
    \by'(p):=\mathrm{diag}(y'_i(p)\mid i \in I) 
\end{align*}
and $F_{k,\pm}(B) \in \mathrm{Mat}_{I\times I}(\bZ)$ is an integral matrix obtained from $B$: see \cite[(3.19)]{Mizuno:Jac}.
\end{prop}

\begin{rem}\label{rem:GHK}
The mutation loops form a group $\Gamma_{[\bi]}$ called the cluster modular group, which acts on the cluster Poisson variety $\X_{[\bi]}$ associated with the mutation class of $\bi$ \cite{FG09}. Each seed $\bi'$ gives rise to a toric birational chart $(\bG_m)^I \cong T_{\bi'} \subset \X_{[\bi]}$, and the action of $\phi \in \Gamma_{[\bi]}$ permutes them. The birational action $\phi: (\bG_m)^I \dashrightarrow (\bG_m)^I$ considered above is the restriction of this action to $T_{\bi}$. 

In general, it is possible that the set $\Fix_\phi$ miss other fixed points in $\X_{[\bi]} \setminus T_{\bi}$, but we do not pursue them in this paper for simplicity of the algebraic discussion. As noted in \cite[Remark 2.6]{GHK}, it is not clear if $\X_{[\bi]}$ is Noetherian or not. In practice, one can consider a subscheme $\X_{[\bi]}^\mathrm{ft}$ \emph{loc.~cit.}~for a better approximation of $\X_{[\bi]}$.
\end{rem}

\begin{dfn}[Primitivity]\label{def:primitive}
Given a mutation loop $\phi$ and a positive integer $m \in \bZ_{>0}$, define the iterated mutation loop $\phi^m$ by periodically repeating the sequence of mutations and permutations defining $\phi$. 
We say that a mutation loop $\phi$ is \emph{primitive} if it does not admit an expression $\phi=\psi^m$ with a mutation loop $\psi$ and $m \geq 2$.
\end{dfn}
For example, we have $\phi^2=\mu_2\mu_1$ in \cref{ex:A2_loop}.

\subsection{Mutation networks}
The \emph{mutation network} introduced by Terashima--Yamazaki \cite{TY} is a bipartite graph that captures the structure of the fixed point equation from a 3D perspective. The associated \emph{Neumann--Zagier matrices} are introduced by Mizuno \cite{Mizuno:Jac}.

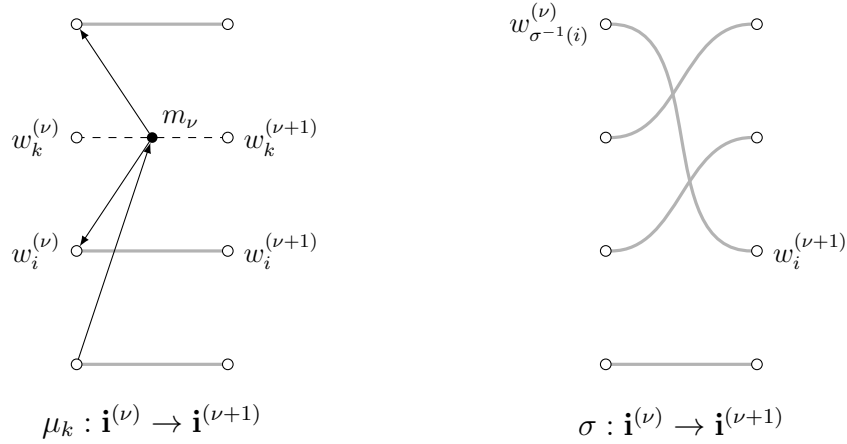
\begin{figure}[ht]
    \centering
\begin{tikzpicture}
\def\R{1.5}
\begin{scope}[rotate=-90]
\foreach \i in {0,2,3}
\draw[gray!60,very thick,shorten >=2pt,shorten <=2pt] (\R*\i,0) -- (\R*\i,2);
\draw[dashed] (\R,0)--(\R,2);
\qarrow{\R,1}{0,0};
\qarrow{\R,1}{2*\R,0};
\qarrow{3*\R,0}{\R,1};
\foreach \i in {0,1,2,3} {
    \filldraw[fill=white](\R*\i,0) circle(2pt);
    \filldraw[fill=white](\R*\i,2) circle(2pt);
    }
\fill (\R,1) circle(2pt);
\node[left=0.2em,scale=0.9] at (\R,0) {$w_k^{(\nu)}$}; 
\node[right=0.2em,scale=0.9] at (\R,2) {$w_k^{(\nu+1)}$};
\node[above right,scale=0.9] at (\R,1) {$m_\nu$}; 
\node[left=0.2em,scale=0.9] at (2*\R,0) {$w_i^{(\nu)}$}; 
\node[right=0.2em,scale=0.9] at (2*\R,2) {$w_i^{(\nu+1)}$};
\node at (3.5*\R,1) {$\mu_k: \bi^{(\nu)} \to \bi^{(\nu+1)}$};
\end{scope}

\begin{scope}[xshift=7cm,rotate=-90]
\draw[gray!60,very thick,shorten >=2pt,shorten <=2pt] (0,0) to[out=90,in=-90] (2*\R,2);
\draw[gray!60,very thick,shorten >=2pt,shorten <=2pt] (\R,0) to[out=90,in=-90] (0,2);
\draw[gray!60,very thick,shorten >=2pt,shorten <=2pt] (2*\R,0) to[out=90,in=-90] (\R,2);
\draw[gray!60,very thick,shorten >=2pt,shorten <=2pt] (3*\R,0) -- (3*\R,2);
\foreach \i in {0,1,2,3} {
    \filldraw[fill=white](\R*\i,0) circle(2pt);
    \filldraw[fill=white](\R*\i,2) circle(2pt);
    }
\node[left=0.2em,scale=0.9] at (0,0) {$w_{\sigma^{-1}(i)}^{(\nu)}$}; 
\node[right=0.2em,scale=0.9] at (2*\R,2) {$w_i^{(\nu+1)}$};
\node at (3.5*\R,1) {$\sigma: \bi^{(\nu)} \to \bi^{(\nu+1)}$};
\end{scope}
\end{tikzpicture}
    \caption{A step-by-step construction of the mutation prenetwork $\widetilde{\cN}_\phi$.}
    \label{fig:construction_prenetwork}
\end{figure}

Let $\phi$ be a mutation loop as in \eqref{eq:mutation_seq}. 
We first extend it to be an infinite periodic sequence by setting $B^{(\nu+mL)}:=B^{(\nu)}$ for $\nu=0,\dots,L$ and $m \in \bZ$, and by periodically repeating the sequence of mutations and permutations. We only need the dynamics of exchange matrices in the following construction, and do not refer to the $y$-variables. 

The \emph{mutation pre-network} $\widetilde{\cN}_\phi$ associated with $\phi$ is the infinite bipartite graph with white/black vertices\footnote{We follow the original coloring convention in \cite{TY}.}, constructed as follows. See \cref{fig:construction_prenetwork}.
\begin{itemize}
    \item We begin with a disjoint collection $\widetilde{W}=\{w_i^{(\nu)}\}_{\nu \in \bZ,~i \in I}$ of white vertices.
    \item If the step $\bi^{(\nu)} \to \bi^{(\nu+1)}$ is a mutation at $k \in I$, then we create a black vertex $m_\nu$ connected with the white vertices $w_{k}^{(\nu)}$ and $w_{k}^{(\nu+1)}$ by a dashed line, respectively. 
    For each $i \neq k$, we connect $w_i^{(\nu)}$ with $w_i^{(\nu+1)}$ by a \textcolor{gray}{gray} line. We draw $|b_{ki}^{(\nu)}|$ arrows $m_\nu \to w_i^{(\nu)}$ if $b_{ki}^{(\nu)}\geq 0$, and $|b_{ki}^{(\nu)}|$ arrows $m_\nu \leftarrow w_i^{(\nu)}$ if $b_{ki}^{(\nu)}< 0$. 
    \item If the step $\bi^{(\nu)} \to \bi^{(\nu+1)}$ is a permutation by $\sigma \in \mathfrak{S}_I$, then we connect $w_{\sigma^{-1}(i)}^{(\nu)}$ with $w_i^{(\nu+1)}$ by a \textcolor{gray}{gray} line for each $i \in I$. 
\end{itemize}
Then we obtain an infinite (partially directed) graph $\widetilde{\cN}_\phi$ with vertex set $\widetilde{W}\cup \widetilde{B}$, where $\widetilde{B}$ denotes the set of the black vertices, together with three types of edges: dashed lines, quiver arrows, and gray lines. Note that it is $L$-periodic by construction. We call a maximal concatenation of gray lines a \emph{vertex orbit}. Let $(b_{m,w})_{(m,w) \in \widetilde{B}\times \widetilde{W}}$ denote the signed adjacency matrix counting the quiver arrows.

The \emph{mutation network} $\cN_\phi$ associated with $\phi$ is obtained from $\widetilde{\cN}_\phi$ by identifying $w_i^{(mL)}=w_i^{(0)}$ for all $i \in I$ and $m \in \bZ$, and contracting each vertex orbit into a point, identifying the white vertices along it. Let $W$ (resp. $B$) denote the set of white (resp. black) vertices in $\cN_\phi$, and $\pi:\widetilde{W} \cup \widetilde{B} \to W \cup B$ the projection of vertices.

\begin{dfn}\label{def:NZ}
Let $\cN_\phi$ be the mutation network associated with a mutation loop $\phi$. 
\begin{enumerate}
    \item Define $W \times B$ integral matrices $N^\ve$, $\ve=0,+,-$, called the \emph{adjacency matrices} by 
\begin{align*}
    N^0_{w,m} := \# \{ \text{dashed edges between $w$ and $m$} \}, \quad 
    N^\pm_{w,m}:= \sum_{\widetilde{w}\in \pi^{-1}(w)} [\pm b_{\widetilde{m},\widetilde{w}}]_+,
\end{align*}
where $\widetilde{m} \in \pi^{-1}(m)$. 
The integral matrices $A^+:=N^0-N^+$, $A^-:=N^0-N^-$, $A^0:=N^0$ are called the \emph{Neumann--Zagier matrices} of $\cN_\phi$. 
    \item The \emph{gluing variety} is the affine scheme
\begin{align*}
    \mathrm{Glue}_\phi:= \{\{(z_{m,+},z_{m,-})\}_{m \in B} \in (\bG_m\times \bG_m)^B \mid z_+ + z_- =1, ~ z_+^{-A^+}z_-^{A^-} =1 \}.
\end{align*}
Here, the defining equations (called the \emph{gluing equations}) are abbreviations of $z_{m,+}+z_{m,-} =1$ for all $m \in B$ and $\prod_{m \in B} z_{m,+}^{-A_{w,m}^+} z_{m,-}^{A_{w,m}^-} =1$ for all $w \in W$, respectively.
\end{enumerate}
\end{dfn}
The coordinate ring of the gluing variety is given by
\begin{align}\label{eq:gluing_variety_ring}
    \cO(\mathrm{Glue}_\phi) = \bQ[z_{m,+}^{\pm 1}, z_{m,-}^{\pm 1} \mid m \in B]/(z_+ + z_- =1, ~ z_+^{-A^+}z_-^{A^-} =1 ).
\end{align}
Let $\Fix_\phi^\ast \subset \Fix_\phi$ be the open subscheme defined by $y_{k_\nu}^{(\nu)}\neq -1$ for each mutation step $\mu_{k_\nu}: \bi^{(\nu)} \to \bi^{(\nu+1)}$ in $\phi$.

\begin{thm}[{\cite[Proposition 3.8]{Mizuno:Jac}}]\label{thm:fix_glue}
We have a morphism of schemes
\begin{align*}
    \varphi: \Fix_\phi^\ast \to \mathrm{Glue}_\phi, \quad \varphi^\ast (z_{m,+},z_{m,-}):= \bigg( \frac{y_m}{1+y_m}, \frac{1}{1+y_m} \bigg).
\end{align*}
If moreover $\phi$ is fully-mutating, then $\varphi$ is an isomorphism.
\end{thm}
We call $\varphi^\ast z_{m,\pm}$, viewed as a function on $\Fix_\phi^\ast$, the \emph{shape parameters}. The isomorphism $\varphi$ provides a useful change of variables in the study of $\Fix_\phi^\ast$.

\begin{rem}
In the original proof of \cite[Proposition 3.8]{Mizuno:Jac}, the fully-mutating condition is assumed in the construction of the map $\varphi$. The construction is valid without the fully-mutating assumption, as follows. It suffices to consider a vertex $w \in W$ which is never mutated. It belongs to a closed gray orbit. By the mutation at $m \in B$, the variables on the neighboring white vertices are transformed as $y'_i=y_i z_{m,+}^{[b_{ki}]_+} z_{m,-}^{-[-b_{ki}]_+}$. The fixed point condition $\by^{(L)}=\by^{(0)}$ then implies that 
\begin{align*}
    \prod_{m \in B} z_{m,+}^{N_{w,m}^+} z_{m,-}^{-N_{w,m}^-} =1.
\end{align*}
Since $N_{w,m}^0=0$ for all $m \in B$, this is equivalent to the gluing equation
\begin{align*}
    \prod_{m \in B} z_{m,+}^{-A_{w,m}^+} z_{m,-}^{A_{w,m}^-} =1.
\end{align*}
\end{rem}

\begin{ex}[Type $A_2$]\label{ex:A2_network}
Consider the mutation loop $\phi$ in \cref{ex:A2_loop}. The quiver $Q=\tikz{\draw(0,0) circle(2pt)node[above,scale=0.8]{$1$};\draw(1,0) circle(2pt)node[above,scale=0.8]{$2$};\qarrow{1,0}{0,0};}$ corresponds to the initial seed $\bi$. 
The mutation network $\cN_{\phi}$ is obtained as follows:
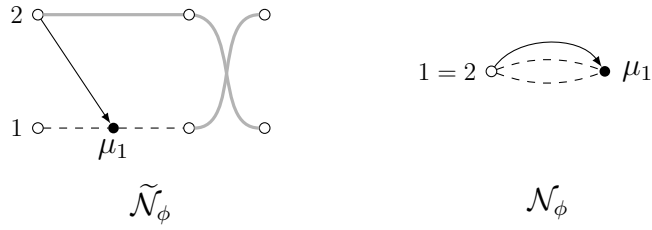
\begin{align*}
\begin{tikzpicture}
\draw(0,0) circle(2pt) node[left=0.2em,scale=0.8]{$1$};
\draw(0,1.5) circle(2pt) node[left=0.2em,scale=0.8]{$2$};
\fill(1,0) circle(2pt) node[below]{$\mu_1$};
\draw(2,0) circle(2pt);
\draw[dashed,shorten >=2pt,shorten <=2pt] (0,0) -- (1,0);
\draw[dashed,shorten >=2pt,shorten <=2pt] (1,0) -- (2,0);
\qarrow{0,1.5}{1,0};
\draw(3,0) circle(2pt);
\draw(3,1.5) circle(2pt);
\draw[gray!60,very thick,shorten >=2pt,shorten <=2pt] (2,0) to[out=0,in=180] (3,1.5);
\draw[gray!60,very thick,shorten >=2pt,shorten <=2pt] (0,1.5) --++(2,0) to[out=0,in=180] (3,0);
\filldraw[fill=white] (2,1.5) circle(2pt);
\node at (1.5,-1) {$\widetilde{\cN}_\phi$};
\begin{scope}[xshift=6cm,yshift=0.75cm]
\draw(0,0) circle(2pt) node[left=0.2em,scale=0.8]{$1=2$};
\fill(1.5,0) circle(2pt) node[right=0.2em]{$\mu_1$};
\draw[dashed,shorten >=2pt,shorten <=2pt] (0,0) to[bend left=20] (1.5,0);
\draw[dashed,shorten >=2pt,shorten <=2pt] (0,0) to[bend right=20] (1.5,0);
\draw[qarrow,shorten >=2pt,shorten <=2pt] (0,0) to[bend left=55] (1.5,0);
\node at (0.75,-1.75) {$\cN_\phi$};
\end{scope}
\end{tikzpicture}
\end{align*}
Here, the vertices connected by \textcolor{gray!70}{gray} lines are identified in $\cN_\phi$. The adjacency and Neumann--Zagier matrices are computed as
\begin{align*}
    &N^0=(2),\quad N^+=(0), \quad N^-=(1), \\
    &A^+=(2), \quad A^-=(1).
\end{align*}
The gluing equations are $z_++z_-=1$ and $z_+^{-2}z_-=1$, which can be written as 
\begin{align*}
    z^{-1}=(1-z)^{2}
\end{align*}
with $z:=z_-$. By the change of variable $z=1/(1+y)$, this equation is equivalent to $1+y=y^2$, which is the defining equation of $R_\phi$. In this case, $R_\phi=R_\phi^\ast$ since $y$ cannot take the value $-1$. 
\end{ex}

\begin{ex}[Pseudo-Anosov on $\Sigma_1^1$]
Consider the initial seed with the exchange matrix $\ve=\begin{psmallmatrix}
    0 & -2 & 2 \\
    2 & 0 & -2 \\
    -2 & 2 & 0
\end{psmallmatrix}$, which corresponds to 
the Markov quiver
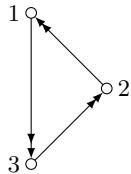
\begin{align*}
\begin{tikzpicture}
    \draw(0,0) circle(2pt) node[left]{\scriptsize $3$};
    \draw(0,2) circle(2pt) node[left]{\scriptsize $1$};
    \draw(1,1) circle(2pt) node[right]{\scriptsize $2$};
    \qqarrow{0,0}{1,1};
    \qqarrow{1,1}{0,2};
    \qqarrow{0,2}{0,0};
\end{tikzpicture}
\end{align*}
This is also the quiver associated with an ideal triangulation of the once-punctured torus $\Sigma_1^1$ \cite{FST}. Consider the mutation loop $\phi=(1\ 2\ 3)\mu_1 \mu_3$, which represents a pseudo-Anosov mapping class on $\Sigma_1^1$. The mutation network is obtained as follows:
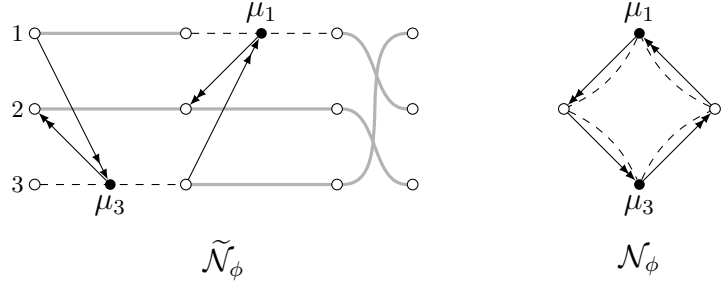
\begin{align*}
\begin{tikzpicture}
    \foreach \i in {(0,1),(0,2),(2,0),(2,1)} \draw[gray!60,very thick,shorten >=2pt,shorten <=2pt] \i --++(2,0);
    \draw[gray!60,very thick,shorten >=2pt,shorten <=2pt] (4,2) to[out=0,in=180] (5,1);
    \draw[gray!60,very thick,shorten >=2pt,shorten <=2pt] (4,1) to[out=0,in=180] (5,0);
    \draw[gray!60,very thick,shorten >=2pt,shorten <=2pt] (4,0) to[out=0,in=180] (5,2);
    \draw[dashed,shorten >=2pt,shorten <=2pt] (0,0) -- (2,0);
    \draw[dashed,shorten >=2pt,shorten <=2pt] (2,2) -- (4,2);
    \draw(0,0) circle(2pt) node[left]{\scriptsize $3$};
    \draw(0,2) circle(2pt) node[left]{\scriptsize $1$};
    \draw(0,1) circle(2pt) node[left]{\scriptsize $2$};
    \fill(1,0) circle(2pt) node[below]{$\mu_3$};
    \foreach \i in {0,1,2} \draw(2,\i) circle(2pt);
    \fill(3,2) circle(2pt) node[above]{$\mu_1$};
    \foreach \i in {0,1,2} \draw(4,\i) circle(2pt);
    \foreach \i in {0,1,2} \draw(5,\i) circle(2pt);
    \qqarrow{0,2}{1,0};
    \qqarrow{1,0}{0,1};
    \qqarrow{3,2}{2,1};
    \qqarrow{2,0}{3,2};
    \node at (2.5,-1) {$\widetilde{\cN}_\phi$};
    \begin{scope}[xshift=7cm]
    \draw(0,1) circle(2pt);
    \draw(2,1) circle(2pt);
    \fill(1,2) circle(2pt) node[above]{{$\mu_1$}};
    \fill(1,0) circle(2pt) node[below]{{$\mu_3$}};
    \qqarrow{0,1}{1,0};
    \qqarrow{1,0}{2,1};
    \qqarrow{2,1}{1,2};
    \qqarrow{1,2}{0,1};
    \draw[dashed,shorten >=2pt,shorten <=2pt] (0,1) to[bend left=20] (1,0);
    \draw[dashed,shorten >=2pt,shorten <=2pt] (1,0) to[bend left=20] (2,1);
    \draw[dashed,shorten >=2pt,shorten <=2pt] (2,1) to[bend left=20] (1,2);
    \draw[dashed,shorten >=2pt,shorten <=2pt] (1,2) to[bend left=20] (0,1);
    \node at (1,-1) {$\cN_\phi$};
    \end{scope}
\end{tikzpicture}
\end{align*}
The adjacency and Neumann--Zagier matrices are computed as
\begin{align*}
    &N^0=\begin{pmatrix}
        1 & 1 \\ 1 & 1
    \end{pmatrix}, \quad 
    N^+=\begin{pmatrix}
        0 & 0 \\ 2 & 2
    \end{pmatrix}, \quad 
    N^-=\begin{pmatrix}
        2 & 2 \\ 0 & 0
    \end{pmatrix}, \\
    &A^+=\begin{pmatrix}
        1 & 1 \\ -1 & -1
    \end{pmatrix}, \quad 
    A^-=\begin{pmatrix}
        -1 & -1 \\ 1 & 1
    \end{pmatrix}.
\end{align*}
In this case, $\dim \mathrm{Glue}_\phi=1$. Therefore $\dim \Fix_\phi=1$ by \cref{thm:fix_glue}, and hence $\phi$ is degenerate in this case. In the original conjecture of Gang--Kim--Yoon \cite{GKY}, this degeneracy is compensated by imposing the additional transverse condition fixing the trace of the monodromy along the meridian.
\end{ex}

\subsection{Refined Neumann--Zagier matrices}\label{subsec:graded_NZ}
Mizuno defined a refined version of Neumann--Zagier matrices \cite{Mizuno:thesis}, including an extra parameter $z$ encoding the ``time difference'' as a grading. We will use the symbol $t$ instead of $z$. 

We say that $\phi$ is \emph{fully-mutating} if each vertex orbit is a finite path. 
In this case, the vertex orbit of each $\widetilde{w} \in \widetilde{W}$ has a unique future endpoint $\partial_+ \widetilde{w} \in \widetilde{W}$, namely the endpoint in the positive direction in $\bZ$.

In the step-by-step construction of the mutation pre-network $\widetilde{\cN}_\phi$, we set the length of each gray segment in each mutation step (the left of \cref{fig:construction_prenetwork}) to be $1$, 
and in each permutation step (the right of \cref{fig:construction_prenetwork}) to be $0$. For each $\widetilde{w} \in \widetilde{W}$, let $\ell_+(\widetilde{w}) \in \bZ_{\geq 0}$ denote the length in this sense of the interval between $\widetilde{w}$ and $\partial_+ \widetilde{w}$ in its vertex orbit. In \cite{Mizuno:thesis}, $\ell_+(\widetilde{w})$ is called the \emph{latency}. 

For each $\widetilde{m}\in \widetilde{B}$, let $\widetilde{m}_-, \widetilde{m}_+ \in \widetilde{W}$ denote the past and the future vertex connected to $\widetilde{m}$ by a dashed line, respectively. If $\widetilde{m}=m_\nu$, then $\widetilde{m}_-=w_k^{(\nu)}$ and $\widetilde{m}_+=w_k^{(\nu+1)}$.

For $(w,m) \in W \times B$, define \cite[(2.16), (2.17)]{Mizuno:thesis}
\begin{align*} 
    N^0(t)_{w,m}:=&\delta_{\pi(\widetilde{m}_-),w} + \delta_{\pi(\widetilde{m}_+),w}t^{1+\ell_+(\widetilde{m}_+)}, \\
    N^\pm(t)_{w,m}:=&\sum_{\widetilde{w} \in \pi^{-1}(w)}[\pm b_{\widetilde{m},\widetilde{w}}]_+ t^{\ell_+(\widetilde{w})},
\end{align*}
where $\widetilde{m} \in \pi^{-1}(m)$. They do not depend on the choice of $\widetilde{m}$. The matrices $A^\pm(t):=N^0(t)-N^\pm(t)$, $A^0(t):=N^0(t) \in \mathrm{Mat}_{W\times B}(\bZ[t])$ are called the \emph{graded Neumann--Zagier matrices} of $\cN_\phi$. We have a remarkable symmetry
\begin{align*}
    A^+(t) A^-(t)^\dagger = A^-(t) A^+(t)^\dagger
\end{align*}
called the \emph{symplectic relation} \cite[(3.2)]{Mizuno:thesis}, where $M(t)^\dagger := M(t^{-1})^\top$ for $M(t) \in \mathrm{Mat}_{m,n}(\bZ[t])$.

\begin{ex}[Type $A_2$]
We continue \cref{ex:A2_network}. From the picture for $\widetilde{\cN}_\phi$, we read off
\begin{align*}
    N^0(t)=(1+t^2), \quad N^+(t)=(0), \quad N^-(t)=(t),
\end{align*}
and hence $A^+(t)=1+t^2$, $A^-(t)=1-t+t^2$. The product $A^+(t)A^-(t)^\dagger = t^2-t+2-t^{-1}+t^{-2}$ is $\dagger$-symmetric. 
\end{ex}

\section{Index for a mutation loop and its relatives}\label{sec:index}

\subsection{Definition of the index}
Let $\phi$ be a mutation loop. 
For each closed point $[x] \in \Fix_\phi$, consider the tangent map
\begin{align}
    (d\phi)_{[x]}: T_{[x]} (\bG_m)^I \to T_{[x]} (\bG_m)^I,
\end{align}
which is a linear map over the residue field $\bQ([x])$. 

\begin{dfn}\label{def:non-deg}
We say that $\phi$ is \emph{non-degenerate} if $\det(1-(d\phi)_{[x]}) \neq 0$ for all closed points $[x] \in \Fix_\phi$. 
\end{dfn}
Consider the $\overline{\bQ}$-points of the fixed point locus:
\begin{align}
    \Fix_\phi(\overline{\bQ})= \{ x \in U_\phi(\overline{\bQ}) \mid \phi(x)=x\},
\end{align}
For each $x \in \Fix_\phi(\overline{\bQ})$, the tangent map 
\begin{align*}
    (d\phi)_x: T_x (\overline{\bQ} \setminus\{0\})^I \to T_x (\overline{\bQ} \setminus\{0\})^I
\end{align*}
is a $\overline{\bQ}$-linear map obtained from the tangent map $(d\phi)_{[x]}$ at the corresponding closed point $[x] \in \Fix_\phi$ by the base change along $\bQ([x]) \to \overline{\bQ}$. Therefore the non-degeneracy is equivalent to the condition that $\det(1-(d\phi)_x) \neq 0$ for all $x \in \Fix_\phi(\overline{\bQ})$. In other words, $(d\phi)_x$ does not have eigenvalue $1$ over $\overline{\bQ}$. 

\begin{dfn}[Index]\label{def:index}
For a non-degenerate mutation loop $\phi$ and any integer $g \geq 0$, 
define 
\begin{align*}
    \Ind_g^\phi:= \sum_{x \in \Fix_\phi(\overline{\bQ})} \det(1-(d\phi)_x)^{g-1} \in \overline{\bQ}.
\end{align*}
We call $\Ind_g^\phi$ the \emph{index} for the mutation loop $\phi$. 
\end{dfn}

\begin{rem}\label{rem:torsion}
Let $\phi$ be any mutation loop. For $x \in \Fix_\phi(\overline{\bQ})$, let $\nu(x) \geq 0$ denote the multiplicity of the eigenvalue $1$ of $(d\phi)_x$. 
In view of Kitayama--Terashima \cite[Corollary 3.11 and Theorem 4.3]{KT}, the value
\begin{align*}
    \left.\frac{\det(t-(d\phi)_x)}{(t-1)^{\nu(x)}}\right|_{t=1} \in \overline{\bQ}
\end{align*}
is a cluster-theoretic analogue of the non-acyclic adjoint Reidemeister torsion at $x$. If $\phi$ is non-degenerate, then $\nu(x)=0$ at every $x \in \Fix_\phi(\overline{\bQ})$, and hence $\Ind_g^\phi$ can be viewed as the power sum of the torsions. 
\end{rem}

\begin{ex}\label{ex:ind_1}
The index for $g=1$ just counts the fixed points:
\begin{align*}
    \Ind_1^\phi=\sum_{x \in \Fix_\phi(\overline{\bQ})} 1 = |\Fix_\phi(\overline{\bQ})|,
\end{align*}
hence an integer.
\end{ex}

\begin{ex}[Type $A_2$]\label{ex:A2_index}
Consider the mutation loop $\phi$ as in \cref{ex:A2_loop}. Then 
\begin{align*}
    \Fix_\phi(\overline{\bQ}) = \left\{ \mtx{y_1 \\y_2}=\mtx{y \\ y^{-1}} ~\middle|~ y^2-y-1=0 \right\}
\end{align*}
consists of two points, corresponding to $y=(1\pm \sqrt 5)/2$. In particular, all the fixed points are defined over the number field $\bQ(\sqrt 5)$.
The tangent map is directly computed from \eqref{eq:A2_rational} as
\begin{align*}
    (d\phi)_x = \mtx{y_2 & 1+y_1 \\ -y_1^{-2} & 0} = \mtx{y^{-1} & y^2 \\ -y^{-2} & 0}.
\end{align*}
Then
\begin{align*}
    \det(1-(d\phi)_x) = \det\mtx{1-y^{-1} & -y^2 \\ y^{-2} & 1} = 2-y^{-1} = \frac{5 \mp \sqrt 5}{2}.
\end{align*}
Hence we get
\begin{align*}
    \Ind_g^\phi = \bigg(\frac{5 + \sqrt 5}{2}\bigg)^{g-1} + \bigg(\frac{5 - \sqrt 5}{2}\bigg)^{g-1}.
\end{align*}
In this case, $a:=(5+\sqrt 5)/2$ and $b:=(5-\sqrt 5)/2$ are two solutions of $\lambda^2-5\lambda+5=0$. We claim that the recurrence
\begin{align*}
    \Ind_g^\phi - 5 \Ind_{g-1}^\phi + 5\Ind_{g-2}^\phi=0
\end{align*}
holds. Indeed, $a^2-5a+5=0$ implies $a^n-5a^{n-1}+5a^{n-2}=0$. Similarly we get $b^n-5b^{n-1}+5b^{n-2}=0$, and the term-wise sum gives the claimed recurrence. This mechanism is well-known as Newton's identities. 
From this recurrence relation, we obtain the integer sequence
\begin{align*}
    (\Ind_g^\phi)_{g \geq 0} = 1,\ 2,\ 5,\ 15,\ 50,\ 175,\ 625,\dots.
\end{align*}
\end{ex}

\subsubsection{Refined index}

\begin{dfn}[Refined index]\label{def:refined_index}
For a non-degenerate mutation loop $\phi$ and an integer $g \geq 0$, define
\begin{align*}
    \Ind_{g,t}^\phi:= \sum_{x \in \Fix_\phi(\overline{\bQ})} \det(t-(d\phi)_x)^{g-1} \in \overline{\bQ}(t).
\end{align*}
We call $\Ind_{g,t}^\phi$ the \emph{refined index} for the mutation loop $\phi$.
\end{dfn}
We have $\Ind_{g,1}^\phi=\Ind_g^\phi$.

\begin{ex}[Type $A_2$]\label{ex:A2_refined}
By a simple modification of the computations in \cref{ex:A2_index}, we get
\begin{align*}
    \Ind_{g,t}^\phi = a(t)^{g-1} + b(t)^{g-1}
\end{align*}
with $a(t):=t^2 - \frac{-1-\sqrt 5}{2}t + 1$ and $b(t):=t^2 - \frac{-1+\sqrt 5}{2}t + 1$. Since $a(t)+b(t)=2t^2 +t+2$ and $a(t)b(t)=t^4+t^3+t^2+t+1$, we get a recursion relation 
\begin{align*}
    \Ind_{g,t}^\phi - (2t^2+t+2)\Ind_{g-1,t}^\phi + (t^4+t^3+t^2+t+1)\Ind_{g-2,t}^\phi =0. 
\end{align*}
We then obtain the sequence of polynomials (a rational function only for $g=0$) with integer coefficients as in \cref{tab:A2_poly}. 
\begin{table}[ht]
    \centering
    \begin{tabular}{c|l}
       $g$  & $\Ind_{g,t}^\phi$  \\ \hline
       $0$  & $(2t^2+t+2)/(t^4+t^3+t^2+t+1)$\\
       $1$  & $2$\\
       $2$  & $2+t+2t^2$\\
       $3$  & $2 + 2t + 7t^2 + 2t^3 + 2t^4$\\
       $4$  & $2 + 3 t + 15 t^2 + 10 t^3 + 15 t^4 + 3 t^5 + 2 t^6$\\
       $5$  & $2 + 4 t + 26 t^2 + 28 t^3 + 55 t^4 + 28 t^5 + 26 t^6 + 4 t^7 + 2 t^8$\\
       $6$  & $2 + 5 t + 40 t^2 + 60 t^3 + 145 t^4 + 121 t^5 + 145 t^6 + 60 t^7 + 
 40 t^8 + 5 t^9 + 2 t^{10}$ \vspace{2mm}
    \end{tabular}
    \caption{Refined indices for the mutation loop of type $A_2$.}
    \label{tab:A2_poly}
\end{table}
\end{ex}

\subsubsection{Generating function of the indices}

\begin{dfn}[Generating function]
For a non-degenerate mutation loop $\phi$, define
\begin{align*}
    Z_\phi(s):= \sum_{g \geq 2} (-1)^g\Ind_g^\phi\cdot s^{g-2} \in \overline{\bQ}[[s]].
\end{align*}
We call $Z_\phi(s)$ the \emph{generating function} of the indices.
\end{dfn}

\begin{lem}
For each $x \in \Fix_\phi(\overline{\bQ})$, let $b_x:=\det(1-(d\phi)_x)$ and consider $E_\phi(s):=\prod_{x \in \Fix_\phi(\overline{\bQ})} (1+b_x s)$. Then we have
\begin{align*}
    Z_\phi(s) =  \frac{E'_\phi(s)}{E_\phi(s)}.
\end{align*}
\end{lem}

\begin{proof}
By applying the Leibniz rule and expanding the geometric series, we get
\begin{align*}
    \frac{E'_\phi(s)}{E_\phi(s)} &= \sum_{x \in \Fix_\phi(\overline{\bQ})} \frac{b_x}{1+b_x s} \\
    &= \sum_{x \in \Fix_\phi(\overline{\bQ})} b_x(1-b_x s + b_x^2 s^2 - b_x^3 s^3 +\cdots) \\
    &= \Ind_2^\phi - \Ind_3^\phi s + \Ind_4^\phi s^2 - \Ind_5^\phi s^3 +\cdots = Z_\phi(s).
\end{align*}
\end{proof}

\begin{ex}[Type $A_2$]\label{ex:A2_generating}
From the computation in \cref{ex:A2_index}, we get 
\begin{align*}
    E_\phi(s)=(1+as)(1+bs) = 5s^2 +5s + 1, \quad Z_\phi(s) = \frac{10s+5}{5s^2 +5s +1}.
\end{align*}
\end{ex}



\section{Integrality of the index}\label{sec:integrality}

\subsection{The index as a trace}\label{subsec:trace}
Recall the $\bQ$-algebra $R_\phi=\cO(\Fix_\phi)$ of regular functions on $\Fix_\phi$. We will henceforth fix a non-degenerate mutation loop $\phi$, and simply write $R:=R_\phi$ during this subsection.

\begin{conv}
A point $x \in \Fix_\phi(\overline{\bQ})$ determines a homomorphism $x^\# \in \Hom_\bQ(R,\overline{\bQ})$, which determines a maximal ideal $\mathfrak{m}_x:=\ker x^\# \subset R$ and a homomorphism $\sigma_x \in \Hom_\bQ(R/{\mathfrak{m}_x},\overline{\bQ})$. The maximal ideal $\mathfrak{m}_x$ corresponds to a closed point $[x] \in \Fix_\phi$. We may regard $x=(\mathfrak{m}_x,\sigma_x)$ or $x=([x],\sigma_x)$.
\end{conv}

\begin{lem}\label{lem:fixed_point_Artinian}
For a closed point $[x] \in \Fix_\phi$, assume $\det(1-d\phi_{[x]}) \neq 0$. 
Then the local ring $R_{[x]}=R_{\mathfrak{m}_x}$
is a field. 
If this is the case for any closed point $[x] \in \Fix_\phi$, namely $\phi$ is non-degenerate, then $R$ is a reduced Artinian ring. In particular, $\dim \Fix_\phi=0$ and $\Fix_\phi$ has only finitely many closed points.
\end{lem}

\begin{proof}
Consider $[x] \in \Fix_\phi \subset (\bG_m)^I$ such that $\det(1-d\phi_{[x]}) \neq 0$ holds. 
By taking an open neighborhood $U$ of $[x]$ if necessary, the fixed point condition $\phi([x])=[x]$ is presented by algebraic equations
\begin{align*}
    y_1 = f_1(y_1,\dots,y_n),\quad \dots, \quad 
    y_n = f_n(y_1,\dots,y_n)
\end{align*}
for some regular functions $f_1,\dots,f_n: U \to \bA^{\!1}_\bQ$. 
Then the Zariski tangent space is identified with 
\begin{align}
    T_{[x]} \Fix_\phi \cong\ker J_{[x]},
\end{align}
where $J_{[x]}:=(\delta_{ij}-\frac{\partial f_i}{\partial y_j}([x]))_{i,j}$ represents the tangent map $1-d\phi_{[x]}$. Since $\det(1-d\phi_{[x]}) \neq 0$ by assumption, we have $T_{[x]} \Fix_\phi=0$.

On the other hand, we have $T_{[x]} \Fix_\phi=(\mathfrak{m}_{R_{[x]}}/\mathfrak{m}_{R_{[x]}}^2)^\ast$, where $\mathfrak{m}_{R_{[x]}}$ is the maximal ideal of $R_{[x]}$. The above argument shows $\mathfrak{m}_{R_{[x]}}=\mathfrak{m}_{R_{[x]}}^2$, and hence $\mathfrak{m}_{R_{[x]}}=0$ by Nakayama's lemma for the $R_{[x]}$-module $\mathfrak{m}_{R_{[x]}}$. Thus $R_{[x]}$ is a field. 

If $\det(1-d\phi_{[x]})\neq 0$ holds for any $[x] \in \Fix_\phi$, then the above argument shows that $R_{[x]}$ is a field for any closed point $[x]$. 
This implies that any prime ideal of $R$ is maximal, and hence $\dim R=0$. Since $R$ is Noetherian, this implies that $R$ is Artinian \cite[Theorem 8.5]{AtiMac}. 
Then $R$ has only finitely many maximal ideals \cite[Proposition 8.3]{AtiMac}, and we have 
\begin{align}\label{eq:Artin_canonical}
    R \cong \prod_{[x] \in \Fix_\phi} R_{[x]}
\end{align}
by the canonical decomposition of an Artinian ring \cite[Theorem 8.7]{AtiMac}. Since each $R_{[x]}$ is a field, we see that $R$ is reduced.
\end{proof}

In particular, $R$ is isomorphic to the finite product of fields $R_{[x]}$ for closed points $[x] \in \Fix_\phi$. Since $R$ is a finitely generated $\bQ$-algebra, each $R_{[x]}/\bQ$ is a finite extension by \cite[Proposition 7.9]{AtiMac}. 

\begin{rem}
The non-degeneracy of $\phi$ is equivalent to the condition (5) in \cite[Tag 02GU]{Stack} with the structure morphism $f: \Fix_\phi \to \Spec \bQ$. Since $R_\phi$ is a finitely generated $\bQ$-algebra, it is reduced Artinian if and only if $\Fix_\phi$ is finite \'etale over $\bQ$.
See \cite[Tag 02GU]{Stack} for related characterizations.
\end{rem}

\begin{rem}\label{rem:complex_fixed_point}
A complex fixed point of $\phi$ corresponds to a $\bQ$-algebra homomorphism $z: R \to \bC$. Since $R$ is finite-dimensional over $\bQ$, each element is algebraic. In particular, $z$ has coordinates algebraic over $\bQ$. Therefore, every complex fixed point of a non-degenerate mutation loop $\phi$ is a $\overline{\bQ}$-point.
\end{rem}

Recall that for a finite extension $L/K$ of fields, we have the trace defined by
\begin{align*}
    \Tr_{L/K}(x) := \Tr(m_x)
\end{align*}
for $x \in L$, where $m_x: L \to L$, $y \mapsto xy$ is the multiplication operator, which is viewed as a $K$-linear map. For a product $L=L_1\times\dots \times L_r$ of finite extensions $L_j/K$, define $\Tr_{L/K}(x):=\sum_{j=1}^r \Tr_{L_j/K}(x)$.

The function $x\mapsto \det(1-d\phi_x)$
defines a regular function on $\Fix_\phi$,
which we denote by 
\begin{align*}
    T:=\det(1-d\phi_{\bullet}) \in R=\cO(\Fix_\phi).
\end{align*}

\begin{prop}\label{prop:index_trace}
For any $g \geq 0$, we have
\begin{align}\label{eq:index_trace}
    \Ind_g^\phi = \Tr_{R/\bQ} (T^{g-1}).
\end{align}
\end{prop}

\begin{proof}
The right-hand side of \eqref{eq:index_trace} is computed as
\begin{align}\label{eq:trace_decomp}
    \Tr_{R/\bQ} (T^{g-1}) = \sum_{[x] \in \Fix_\phi}\Tr_{R_{[x]}/\bQ} (T^{g-1})
\end{align}
by the canonical decomposition of Artinian rings.
Since $\bQ$ has characteristic $0$, each extension $R_p/\bQ$ is separable. It is well-known that for such a finite separable extension, the injective homomorphism $R_{[x]} \to \overline{\bQ}^{\Hom_\bQ(R_{[x]},\overline{\bQ})}$, $a \mapsto (\sigma(a))_{\sigma}$ extends to an isomorphism
\begin{align}\label{eq:isom_diagonalization}
    R_{[x]} \otimes_\bQ \overline{\bQ} \xrightarrow{\sim} \overline{\bQ}^{\Hom_\bQ(R_{[x]},\overline{\bQ})}.
\end{align}
Moreover, the action of $m_T^{g-1} \otimes 1$ is diagonalized in the right-hand side as $(\sigma(T)^{g-1})_{\sigma}$. 

Recall that the closed point $[x]\in \mathrm{Fix}_\bQ$ and a homomorphism $\sigma \in \Hom_\bQ(R_{[x]}, \overline{\bQ})$ together determine a $\overline{\bQ}$-point $x=([x],\sigma) \in \Fix_\phi(\overline{\bQ})$.
Since the trace is invariant under the scalar extension, the sum \eqref{eq:trace_decomp} can be rewritten as
\begin{align}
\begin{aligned}\label{eq:trace_decomp2}
    \Tr_{R/\bQ} (T^{g-1}) &= \sum_{[x] \in \Fix_\phi}\Tr_{R_{[x]}/\bQ} (T^{g-1}) \\
    &= \sum_{[x] \in \Fix_\phi} \sum_{\sigma \in \Hom_\bQ(R_{[x]},\overline{\bQ})}\sigma(T)^{g-1} \\
    &= \sum_{x=([x],\sigma) \in \Fix_\phi(\overline{\bQ})}\sigma(T)^{g-1} \\
    &= \sum_{x \in \Fix_\phi(\overline{\bQ})}\det(1-d\phi_x)^{g-1} = \cI_g^\phi.
\end{aligned}
\end{align}
Thus the assertion is proved.
\end{proof}
From the proof, we obtain the decomposition
\begin{align}\label{eq:Artin_canonical_barQ}
    R \otimes_\bQ \overline{\bQ} \cong \prod_{x \in \Fix_\phi(\overline{\bQ})} \overline{\bQ}e_x,
\end{align}
where $e_x$ is the idempotent element such that $e_x \cdot e_y =\delta_{xy}e_x$ and $1_R=\sum_{x \in \Fix_\phi(\overline{\bQ})}e_x$.

\begin{cor}\label{cor:rational}
Assume that $\phi$ is non-degenerate (\cref{def:non-deg}).
Then $\Ind_g^\phi \in \bQ$ for all $g \geq 0$.
\end{cor}

\begin{proof}
This follows from $\Ind_g^\phi=\Tr_{R/\bQ} (T^{g-1}) \in \bQ$.
\end{proof}

\begin{prop}\label{lem:index_integral}
If the characteristic polynomial of $m_T \in \End_\bQ(R)$ has integral coefficients, then $\Ind_g^\phi \in \bZ$ for all $g \geq 1$.
\end{prop}

\begin{proof}
The assumption implies that all the elementary symmetric polynomials of the eigenvalues of $m_T$ are integers. By Newton's identities, every power sum is an integer polynomial in the elementary symmetric functions, hence an integer as well.
\end{proof}

\begin{rem}\label{rem:I_1_dim}
From \cref{ex:ind_1}, we further obtain
$\Ind_1^\phi = |\Fix_\phi(\overline{\bQ})| = \dim_\bQ R_\phi$.
\end{rem}

\subsection{Refinement}
For the refined index, consider $T(t):=\det(t-d\phi_\bullet) \in R[t]$. The multiplication operator $m_{T(t)}: R[t] \to R[t]$ is defined as a $\bQ[t]$-linear map. \cref{prop:index_trace} is refined to be
\begin{align*}
    \Ind_{g,t}^\phi=\Tr_{R[t]/\bQ[t]}(T(t)^{g-1})
\end{align*}
for $g \geq 1$. For $g=0$, we have $\Ind_{0,t}^\phi=\Tr_{R(t)/\bQ(t)}(T(t)^{-1})$. 
Indeed, the isomorphism \eqref{eq:isom_diagonalization} is extended to be
\begin{align*}
    R_{[x]}[t] \otimes_{\bQ[t]} \overline{\bQ}[t] \cong \overline{\bQ}[t]^{\Hom_{\bQ}(R_{[x]},\overline{\bQ})},
\end{align*}
and the action of $m_{T(t)}^{g-1}\otimes 1$ is diagonalized in the right-hand side as $(\sigma(T(t))^{g-1})_\sigma$. The remaining argument is the same. In particular, we have 
\begin{align*}
    &\Ind_{g,t}^\phi \in \bQ[t] \qquad \mbox{for $g \geq 1$}, \\
    &\Ind_{0,t}^\phi \in \bQ(t),
\end{align*}
as a refinement of \cref{cor:rational}. The final statement (\cref{lem:index_integral}) is refined as follows:
\begin{prop}\label{lem:index_integral_refined}
If the characteristic polynomial of $m_{T(t)} \in \End_{\bQ[t]}(R[t])$ has integral coefficients, then $\Ind_{g,t}^\phi \in \bZ[t]$ for all $g \geq 1$ .
\end{prop}

For an inclusion $R \subset S$ of unital commutative rings, let $\overline{R}^S \subset S$ denote its integral closure of $R$ in $S$, which consists of those elements $a \in S$ satisfying $f(a)=0$ for some monic polynomial $f \in R[X]$. 
We say that $R$ is integrally closed in $S$ if $R=\overline{R}^S$. 
An element $a \in S$ belongs to $\overline{R}^S$ if and only if $S[a]$ is a finitely generated $A$-module \cite[Proposition 5.1]{AtiMac}.

For a finite-dimensional $\bQ$-algebra $R$, let $\cO_R:=\overline{\bZ}^R \subset R$ denote the subring of elements integral over $\bZ$. 
It is well-known that $\bZ[t]$ is integrally closed in $\bQ(t)$. 
Let $\Omega \subset \overline{\bQ}$ denote the ring of algebraic integers. 

\begin{lem}\label{lem:char_poly_integral}
If $T(t) \in \cO_{R_\phi}[t]$, then the characteristic polynomial of $m_{T(t)} \in \End_{\bQ[t]}(R_\phi[t])$ has coefficients in $\bZ[t]$.
\end{lem}

\begin{proof}
Write $T(t)=\sum_{i=0}^d a_i t^i$ with $a_0,\dots,a_d \in \cO_{R_\phi}$. The integrality implies that $S:=\bZ[a_0,\dots,a_d] \subset R_\phi$ is a finitely generated $\bZ$-module, and hence $S[t] \subset R_\phi[t]$ is a finitely generated $\bZ[t]$-module. Hence $T(t)$ is integral over $\bZ[t]$ as an element of $S[t]$. In particular, there exists a monic polynomial $f(X) \in (\bZ[t])[X]$ such that $f(T(t))=0$, and hence $f(m_{T(t)})=0$. 

After extending scalars to $\bQ(t)$, let $\lambda_1,\dots,\lambda_r \in \overline{\bQ(t)}$ be the eigenvalues of $m_{T(t)} \in \End_{\bQ[t]}(R_\phi[t])$, where $r:=\dim_\bQ R_\phi$. The characteristic polynomial of $m_{T(t)}$ is written as
\begin{align*}
    \chi_t(X):=\prod_{i=1}^r (X-\lambda_i) = \sum_{k=0}^r c_k X^k \in (\bQ[t])[X].
\end{align*}
Since each $\lambda_i$ is a root of the monic polynomial $f$, it is integral over $\bZ[t]$. Therefore, the coefficients $c_0,\dots,c_r$ are also integral over $\bZ[t]$. 
Since $\bZ[t]$ is integrally closed, 
it implies that they belong to $\bZ[t]$. Thus $\chi_t(X)$ has coefficients in $\bZ[t]$. 
\end{proof}

\subsection{Integrality in the finite order case}

\begin{thm}\label{thm:finite_order}
For any non-degenerate mutation loop $\phi$ such that $\phi^m=1$ for some $m \in \bZ_{>0}$, we have
\begin{align*}
    \Ind_{g,t}^\phi \in \bZ[t]
\end{align*}
for all $g \geq 1$.
\end{thm}

\begin{proof}
For each $x \in \mathrm{Fix}_\phi(\overline{\bQ})$, $\phi^m=1$ implies $(d\phi)_{x}^m=\mathrm{id}$. In particular, all the eigenvalues of $(d\phi)_{x}$ are $m$-th roots of unity, which are algebraic integers. Therefore, the coefficients of the characteristic polynomial of $(d\phi)_{x}$ belong to $\Omega$.

Write $R:=R_\phi$, and $T(t)=\sum_{i=0}^d a_i t^i \in R[t]$. We can write $a_i=(a_{i,[x]})_{[x] \in \Fix_\phi}$ in the canonical decomposition $R \cong \prod_{[x] \in \Fix_\phi}R_{[x]}$. Every embedding $\sigma: R_{[x]} \to \overline{\bQ}$ gives rise to a geometric point $x=([x],\sigma) \in \Fix_\phi(\overline{\bQ})$. The base-change $\sigma: R_{[x]}[t] \to \overline{\bQ}[t]$ satisfies
\begin{align*}
    \sigma(T(t)_{[x]})=\det(t-(d\phi)_x) \in \overline{\bQ}[t].
\end{align*}
In particular, $\sigma(a_{i,[x]}) \in \Omega$ since it is a coefficient of the characteristic polynomial of $(d\phi)_x$. Since $\sigma$ is an embedding, this implies $a_{i,[x]} \in \cO_{R_{[x]}}$. Therefore $a_i \in \prod_{[x]} \cO_{R_{[x]}} = \cO_R$ for all $i=0,\dots,d$, and hence $T(t) \in \cO_R[t]$. Then by \cref{lem:char_poly_integral} and \cref{lem:index_integral_refined}, we obtain $\Ind_{g,t}^\phi \in \bZ$ for all $g \geq 1$.
\end{proof}

\subsection{Integrality condition from Mizuno's Jacobian formula}\label{subsec:Mizuno}

We obtain the following integrality criterion:

\begin{thm}\label{prop:shape_integral}
Assume that $\phi$ is non-degenerate and $\Fix_\phi=\Fix_\phi^\ast$. 
If moreover the shape parameters $\varphi^\ast z_{m,\pm}$ belong to $\cO_{R_\phi}$ for all $m \in B$, then 
\begin{align*}
    \Ind_{g,t}^\phi \in \bZ[t]
\end{align*}
for all $g \geq 1$. In particular, we have $\Ind_g^\phi \in \bZ$ for all $g \geq 1$. 
\end{thm}

\begin{proof}
Let $x \in \Fix_\phi(\overline{\bQ})$ be a non-degenerate fixed point of $\phi$. We use the Jacobian formula (\cref{prop:Jacobian}) for each mutation step. After the composition, the $\by$-terms cancel together except for the initial and the terminal ones, which coincides with each other by the fixed point condition $\by^{(L)}=\by^{(0)}$. Hence, the Jacobian of $\phi$ at $x$ is conjugate to the product of matrices of the form
\begin{align*}
    \varphi^\ast z_{m,-}(x) F_{k,+}(B)^\top + \varphi^\ast z_{m,+}(x) F_{k,-}(B)^\top
\end{align*}
and permutation matrices. Here, each $\varphi^\ast z_{m,\pm}(x)$ is well-defined by $x\in \Fix_\phi^\ast(\overline{\bQ})$.
Since $\varphi^\ast z_{m,\pm} \in \cO_{R_\phi}$ by assumption, it follows that $(d\phi)_\bullet$ is presented by a matrix with entries in $\cO_{R_\phi}$. Therefore $T(t) \in \cO_{R_\phi}[t]$, and hence we get $\Ind_{g,t}^\phi \in \bZ[t]$ for all $g \geq 1$ by combining \cref{lem:char_poly_integral} and \cref{lem:index_integral_refined}.  
\end{proof}

\begin{rem}
Since $z_{m,+}+z_{m,-}=1$, the integrality of one of them implies the other.
\end{rem}

\subsection{Mizuno's determinant formula}
Based on the Jacobian formula (\cite{Mizuno:Jac}), Mizuno obtained the following closed form expression of the characteristic polynomial of the Jacobian of a mutation loop, which can be used to explicitly compute $T(t)$ in special cases. 

We assume that $\phi$ is fully-mutating, and that $|B|=|W|$ in the mutation network $\cN_\phi$. Fix an enumeration $B\cong \{1,\dots,r\}$, and define $W \cong \{1,\dots,r\}$ by labeling each white vertex by the black vertex at its future end. The rows and columns of (graded) Neumann-Zagier matrices are enumerated in this way.

\begin{thm}[Mizuno {\cite[Theorem 1.1]{Mizuno:Jac} for $t=1$; \cite[Proposition 3.4.3]{Mizuno:thesis_unpub}}]
For a fully-mutating mutation loop $\phi$ such that $r:=|B|=|W|$, we have
\begin{align}\label{eq:Mizuno_det}
    \det(1-t^\ell(d\phi)_x) = \det(A^+(t) \mathbf{z}_-(x) + A^-(t) \mathbf{z}_+(x))
\end{align}
for any $x \in \Fix_\phi^\ast(\overline{\bQ})$. 
Here, 
\begin{itemize}
    \item $A^\pm(t) \in \mathrm{Mat}_{r}(\bZ[t])$ are graded Neumann--Zagier matrices of $\phi$ \cref{subsec:graded_NZ};
    \item $\mathbf{z}_\pm(x):=\mathrm{diag}(z_{m,\pm}(\varphi(x))\mid m \in B)$ (recall \cref{thm:fix_glue});
    \item $\ell \in \bZ_{\geq 0}$ is the total number of mutations in $\phi$.\footnote{In the general formulation of \cite{Mizuno:thesis,Mizuno:thesis_unpub}, this number depends on a choice of the \emph{consistent subset} $R \subset I \times \bZ$, and $\ell$ can be made smaller when $\phi$ contains a block of mutually-commuting mutations. }
\end{itemize}

\end{thm}

\begin{cor}\label{cor:T_NZ_data}
For a fully-mutating mutation loop $\phi$ such that $r=|B|=|W|$, we have
\begin{align*}
    T(t^\ell)|_{\Fix_\phi^\ast} = t^{\ell |I|}\det(A^+(t^{-1}) \mathbf{z}_- + A^-(t^{-1}) \mathbf{z}_+) \in \cO(\Fix_\phi^\ast).
\end{align*}
\end{cor}

\begin{rem}
The formula \eqref{eq:Mizuno_det}
holds without the Cartan-like assumption in \cite{Mizuno:thesis_unpub}. Note that the original formula in \cite[Proposition 3.4.3]{Mizuno:thesis_unpub} is recovered by using symmetry for $t \leftrightarrow t^{-1}$ for the Cartan-like case. The proof is valid in the general case without essential modification: the square matrices $\widetilde{L}$, $\widetilde{H}$ of size $\ell|I|$ are defined in the same way, and we have $\det(1-t\widetilde{L})=\det(1-t^\ell(d\phi)_x)$. The matrix $\widetilde{H}$ is nilpotent by the fully-mutating condition, and $\det(1-t\widetilde{H})=1$. The matrix $(1-t\widetilde{H})^{-1}(1-t\widetilde{L})$ is the identity matrix except for the columns corresponding to the mutation indices, and its submatrix corresponding to the mutation indices coincides with $A^+(t) \mathbf{z}_-(x) + A^-(t) \mathbf{z}_+(x)$. 
\end{rem}

\section{\texorpdfstring{$(1+1)$}{(1+1)}-dimensional interpretation}\label{sec:Frobenius}
Recall the equivalence between the $(1+1)$-dimensional TQFT and commutative Frobenius algebras \cite{Abrams,Kock}. A \emph{Frobenius algebra} over a field $k$ is a finite-dimensional $k$-algebra $R$ equipped with a linear form $\ve: R \to k$ (called the \emph{Frobenius form}) whose kernel contains no non-trivial ideals \cite[Definition 2.2.1]{Kock}. Equivalently, it is characterized by a non-degenerate pairing \cite[2.2.3]{Kock}. 
In this section, we introduce two Frobenius structures on the fixed point algebra $R_\phi=\cO(\Fix_\phi)$ and provide a $(1+1)$-dimensional insight to our indices $\Ind_g^\phi$. 

\subsection{Frobenius structure induced by $\phi$}\label{subsec:Frob_T}
Let $\phi$ be a non-degenerate mutation loop, and consider the commutative algebra $R_\phi^{\overline{\bQ}}:=R_\phi \otimes_\bQ \overline{\bQ}$. Recall from \eqref{eq:Artin_canonical_barQ} the canonical decomposition
\begin{align*}
    R_\phi^{\overline{\bQ}} \cong \prod_{x \in \Fix_\phi(\overline{\bQ})} \overline{\bQ} e_{x}.
\end{align*}
Given an element $F \in R_\phi^{\overline{\bQ}}$, denote by $F=(F(x))_{x \in \Fix_\phi(\overline{\bQ})}$ its decomposition with respect to the canonical decomposition. 
Recall the element $T=\det(1-d\phi_{\bullet}) \in R_\phi$.

\begin{dfn}
Define a linear form $\ve_\phi: R_\phi^{\overline{\bQ}} \to \overline{\bQ}$ by
\begin{align*}
    \ve_\phi(F):= \Tr_{R_\phi^{\overline{\bQ}}/{\overline{\bQ}}}(m_{F/T}) = \sum_{x \in \Fix_\phi(\overline{\bQ})}\frac{F(x)}{T(x)}
\end{align*}
for $F \in R_\phi^{\overline{\bQ}}$.
\end{dfn}

\begin{lem}\label{lem:Frobenius_T}
$\ve_\phi$ is a Frobenius form. In particular, $(R_\phi^{\overline{\bQ}}, \ve_\phi)$ is a commutative Frobenius algebra over $\overline{\bQ}$.
\end{lem}

\begin{proof}
Consider the bilinear form
\begin{align*}
    (F|G)_\phi := \sum_{x \in \Fix_\phi(\overline{\bQ})}\frac{F(x)G(x)}{T(x)}
\end{align*}
for $F,G \in R_\phi^{\overline{\bQ}}$. Then $\ve_\phi(F)=(F|1)_\phi$. Moreover, the pairing is non-degenerate. Indeed, in the basis of idempotents, we have $(e_{x}|e_{y})_\phi=\delta_{xy} T(x)^{-1}$, which shows that the dual basis of $e_{x}$ is $e_{x}^\ast:=T(x)e_{x}$. Therefore, the assertion follows from \cite[Lemma 2.2.4]{Kock}.
\end{proof}
The non-degenerate bilinear form $(-|-)_\phi$ induces an isomorphism $R_\phi^{\overline{\bQ}} \cong (R_\phi^{\overline{\bQ}})^\ast$.

\begin{lem}\label{lem:handle_T}
The element $T \in R_\phi^{\overline{\bQ}}$ gives the handle element \cite[Section 3.3]{Abrams} in the Frobenius algebra $(R_\phi^{\overline{\bQ}},\ve_\phi)$.
\end{lem}

\begin{proof}
The handle element $h$ in a Frobenius algebra $(R,\ve)$ is defined to be
\begin{align*}
    h:=m \circ \Delta (1) \in R,
\end{align*}
where $m: R \otimes R \to R$ is the multiplication, and $\Delta:R \to R \otimes R$ is the comultiplication induced by the duality $R \cong R^\ast$ under the bilinear pairing. In our Frobenius algebra $(R_\phi^{\overline{\bQ}},\ve_\phi)$, it is computed as follows:
\begin{align*}
    h = m \circ \Delta(1) = m \bigg(\sum_{x \in \Fix_\phi} e_{x}\otimes e_{x}^\ast \bigg) = \sum_{x \in \Fix_\phi} e_{x} e_{x}^\ast = \sum_{x \in \Fix_\phi} e_{x} T(x) = T.
\end{align*}
\end{proof}
In the $(1+1)$-dimensional picture, the multiplication of the handle element $h \in R$ corresponds to attaching a torus with two boundary components \cite[Figure 8]{Abrams}, hence the name. We have the genus-$g$ correlator
\begin{align*}
    \langle F \rangle_g:=\ve (h^g F)
\end{align*}
for $F \in R$, 
and the genus-$g$ partition function
\begin{align*}
    Z(\Sigma_g) := \langle 1 \rangle_g.
\end{align*}
The following gives a $(1+1)$-dimensional interpretation of our indices $\Ind_g^\phi$:

\begin{thm}\label{thm:TQFT}
We have 
\begin{align*}
    Z(\Sigma_g) = \Ind_g^\phi
\end{align*}
in the Frobenius algebra $(R_\phi^{\overline{\bQ}},\ve_\phi)$.
\end{thm}

\begin{proof}
The correlators are given by
\begin{align*}
    \langle F \rangle_g = \ve_\phi(T^g F) = \sum_{x \in \Fix_\phi} T(x)^{g-1} F(x)
\end{align*}
for $F \in R_\phi^{\overline{\bQ}}$, and hence the partition functions are
\begin{align*}
    Z(\Sigma_g)= \langle 1 \rangle_g = \sum_{x \in \Fix_\phi} T(x)^{g-1} = \Ind_g^\phi.
\end{align*}
\end{proof}

\subsection{Residual Frobenius structure}\label{subsec:residue}
Following Siebert--Tian \cite[Section 4]{SiebertTian}, we discuss another Frobenius structure on $R_\phi^\bC:=R_\phi \otimes_\bQ \bC$ that depends on its presentation as an Artinian complete intersection. \cref{prop:SiebertTian} below will be useful for the computation of $\Ind_0^\phi$. 

Let $\psi_1,\dots,\psi_k \in \bC[u_1,\dots,u_k]$ such that the holomorphic map $\psi:=(\psi_1,\dots,\psi_k): \bC^k \to \bC^k$ has only finitely many simple zeros. This is equivalent to saying that $R:=\bC[u_1,\dots,u_k]/(\psi_1,\dots,\psi_k)$ is reduced Artinian. 

For $F \in \bC[u_1,\dots,u_k]$ and $x \in \psi^{-1}(0)$, consider the analytic residue
\begin{align*}
    \mathrm{res}_\psi(x;F):=\frac{1}{(2\pi i)^k} \int_{\Gamma^\varepsilon(x)} \frac{F}{\psi_1 \dots \psi_k}du_1\dots du_k,
\end{align*}
where $\Gamma^\varepsilon(x):=\{ z \in U(x) \mid |\psi_i(z)|=\varepsilon\}$ with sufficiently small neighborhood $U(x) \subset \bC^k$ such that $U(x) \cap \psi^{-1}(0)=\{x\}$, and $\varepsilon>0$ is sufficiently small so that $\Gamma^\varepsilon (x)\Subset U(x)$. It is oriented by $d(\mathrm{arg}\psi_1)\wedge\dots\wedge d(\mathrm{arg}\psi_k)|_{\Gamma^\varepsilon(x)}$. Then, the total residue
\begin{align*}
    \mathrm{Res}_\psi(F):= \sum_{x \in \psi^{-1}(0)}\mathrm{res}_\psi(F)
\end{align*}
provides a Grothendieck residue $\mathrm{Res}_\psi: R \to \bC$ in the context of Gorenstein rings. Using the Jacobian $J:=\det(\partial \psi_i/\partial u_j) \in R$, we obtain 
\begin{align*}
    \mathrm{Res}_\psi(F) = \sum_{x \in \psi^{-1}(0)} \frac{F(x)}{J(x)}
\end{align*}
by residue computation after the change of variables $v_i:=\psi_i(u_1,\dots,u_k)$. 

\begin{lem}[cf. {\cite[2.3.7]{Kock}}]\label{lem:Frobenius_J}
$\mathrm{Res}_\psi: R \to \bC$ is a Frobenius form. 
In particular, $(R,\mathrm{Res}_\psi)$ is a commutative Frobenius algebra over $\bC$.
\end{lem}

\begin{proof}
Follows similarly to the proof of \cref{lem:Frobenius_T} by the non-degeneracy of the bilinear pairing
\begin{align*}
    (F|G)_\psi := \sum_{x \in \psi^{-1}(0)} \mathrm{res}_\psi(FG).
\end{align*}
\end{proof}
By a similar argument to \cref{lem:handle_T}, we see that $J$ gives the handle element in the Frobenius algebra $(R,\mathrm{Res}_\psi)$.


\paragraph{\textbf{Weighted Gr\"obner basis}}
Fix a vector $d=(d_1,\dots,d_k) \in \bZ_{>0}^k$, define
\begin{align*}
    \deg_d (u_1^{a_1} \dots u_k^{a_k}) := \sum_{i=1}^k d_i a_i.
\end{align*}
for $f \in \bC[u_1,\dots,u_k]$. Let $\mathrm{in}_d(f)$ denote the highest degree homogeneous part in $f$ (where $\mathrm{in}_d$ stands for ``initial''). 
For an ideal $I \subset \bC[u_1,\dots,u_k]$, we call
\begin{align*}
    \mathrm{in}_d(I):=(\mathrm{in}_d(f) \mid f \in I)
\end{align*}
the \emph{initial ideal}. We call a finite subset $\{\psi_1,\dots,\psi_r\} \subset I$ a \emph{Gr\"obner basis} if $\mathrm{in}_d(I)=(\mathrm{in}_d(\psi_1),\dots,\mathrm{in}_d(\psi_r))$. The initial forms are related to zeros at infinity  as follows. 

Let $\bP^{k}_{(1,d)}$ denote the weighted projective space with degree $d$. For $f \in \bC[u_1,\dots,u_k]$, consider the homogenization
\begin{align*}
    \widetilde{f}(u_0,u_1,\dots,u_k):=u_0^{\deg_d f} f(u_0^{-d_1}u_1,\dots, u_0^{-d_k}u_k).
\end{align*}
Then we have
\begin{align*}
    \widetilde{f}(1,u)=f(u), \quad \widetilde{f}(0,u)=(\mathrm{in}_d(f))(u),
\end{align*}
where the latter means that $(\mathrm{in}_d(f))(u)$ equals the restriction of $\widetilde{f}$ to the hyperplane $\{u_0=0\}$ at infinity of $\bP^{k}_{(1,d)}$. Therefore, the polynomials $\psi_1,\dots,\psi_k$ have no common zeros on the hyperplane $\{u_0=0\}$ at infinity if and only if
\begin{align*}
    V(\mathrm{in}_d(\psi_1),\dots, \mathrm{in}_d(\psi_k))=\{0\} \subset \bA^k.
\end{align*}
Having this geometric intuition, we will use the following:

\begin{prop}[{\cite[Proposition 4.1]{SiebertTian}}]\label{prop:SiebertTian}
Let $R=\bC[u_1,\dots,u_k]/(\psi_1,\dots,\psi_k)$ be a reduced Artinian complete intersection with $\{\psi_1,\dots,\psi_k\}$ a Gr\"obner basis of the defining ideal with respect to a vector $d=(d_1,\dots,d_k)$. Let $N:=\sum_{i=1}^k \deg_d(\psi_i) -\sum_{i=1}^k d_i$, 
\begin{align*}
    R_{<N} := \pi(\{ F \in \bC[u_1,\dots,u_k] \mid \deg_d(F)<N\})
\end{align*}
the image under the canonical projection $\pi: \bC[u_1,\dots,u_k] \to R$, 
and $J:=\det(\partial \psi_i/\partial u_j) \in R$. Then we have 
\begin{align*}
    R= R_{<N} \oplus \bC J,
\end{align*}
and the commutative diagram
\begin{equation*}
\begin{tikzcd}
    \bC[u_1,\dots,u_k] \ar[r,"\mathrm{Res}_\psi"] \ar[d,"\pi"'] & \bC \\
    R \ar[r,"\mathrm{pr}_2"'] & \bC J, \ar[u]
\end{tikzcd}
\end{equation*}
where $\bC J \to \bC$, $J \mapsto \dim_\bC R$.
\end{prop}
In particular, the diagram shows
\begin{align*}
    \mathrm{Res}_\psi(F)=0
\end{align*}
for any $F \in R_{<N}$. This vanishing is a consequence of the residue theorem for the rational differential form $\displaystyle\frac{F}{\psi_1 \dots \psi_k}du_1 \dots du_k$ on $\bP^{k+1}_{(1,d)}$, where the Gr\"obner basis assumption ensures that there are no additional zeros at infinity. If $F=cJ$, then $\mathrm{Res}_\psi(F)=c \dim_\bC R$.

\begin{cor}\label{cor:SiebertTian}
Let $\phi$ be a non-degenerate mutation loop. If the Artinian ring $R_\phi^\bC$ has a presentation as in \cref{prop:SiebertTian}, then we have
\begin{align*}
    \Ind_g^\phi = \mathrm{Res}_\psi(JT^{g-1}).
\end{align*}
In particular, if $J/T \in (R_\phi^\bC)_{<N}$, then $\Ind_0^\phi=0$.
\end{cor}

\begin{proof}
The right-hand side is computed as
\begin{align*}
    \mathrm{Res}_\psi(JT^{g-1}) = \sum_{x \in \psi^{-1}(0)} \frac{J(x)T(x)^{g-1}}{J(x)} = \sum_{x \in \Fix_\phi(\bC)} T^{g-1}(x) = \Ind_g^\phi.
\end{align*}
Here, we used the fact that all the complex $\phi$-fixed points are $\overline{\bQ}$-points (\cref{rem:complex_fixed_point}).
\end{proof}

\begin{rem}
\begin{enumerate}
    \item $\mathrm{Res}_\psi(J)=\dim_\bC R_\phi^\bC$ agrees with $\Ind_1^\phi=\dim_\bQ R_\phi$ (\cref{rem:I_1_dim}).
    \item When applying \cref{cor:SiebertTian}, it is necessary to have a polynomial presentation rather than a \emph{Laurent} polynomial presentation. We will give such a presentation in \cref{sec:example} for explicit examples.
\end{enumerate}
\end{rem}

\section{Examples}\label{sec:example}

\subsection{Mutation loops of period $1$}
Fix an integer $N \geq 2$ and a palindromic sequence $\mathbf{a}=(a_2,\dots,a_N)$ of integers: $a_j=a_{N-j+2}$. Define a skew-symmetric matrix $B=(b_{ij})_{i,j=1}^N$ inductively by
\begin{align*}
    b_{1j}=a_j, \quad b_{ij}=b_{i-1,j-1} + a_i[-a_j]_+ -a_j[-a_i]_+
\end{align*}
for $i,j=2,\dots,N$. Fordy and Marsh \cite{FM11} proved that $\phi:\bi=\bi^{(0)} \xrightarrow{\mu_1} \bi^{(1)} \xrightarrow{(N\cdots 2\ 1)} \bi^{(2)}$ is a mutation loop with initial seed $\bi=(B,(y_1,\dots,y_N))$, and this construction produces all the skew-symmetric matrices $B$ with this property. See \cite[Theorem 4.1]{Mizuno:thesis} for another proof of this classification in terms of the Neumann--Zagier matrices. Such $\phi$ are called the \emph{mutation loops of period $1$}. The induced birational map $\phi: (\bG_m)^N \dashrightarrow (\bG_m)^N$ is computed as
\begin{align}\label{eq:FM_y_action}
    \phi^\ast: \mtx{y_1 \\ y_2 \\ \vdots \\ y_N} \mapsto \mtx{y_2(1+y_1^{-\sgn a_2})^{-a_2} \\ \vdots \\ y_N(1+y_1^{-\sgn a_N})^{-a_N}\\ y_1^{-1}}.
\end{align}
Hence the fixed point equation reduces to 
\begin{align}\label{eq:FM_fixed_y}
    y^2 = \prod_{j=2}^N (1+y^{-\sgn a_j})^{-a_j}
\end{align}
with $y:=y_1$. If $\mathbf{a}\neq \mathbf{0}$, then $y$ cannot take the value $-1$. Therefore $\Fix_\phi =\Fix_\phi^\ast \cong \mathrm{Glue}_\phi$. 

The associated mutation (pre)networks $\widetilde{\cN}_\phi$, $\cN_\phi$ are obtained as follows:
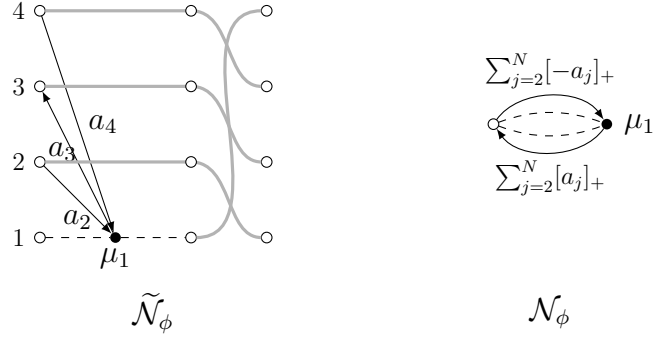
\begin{align*}
\begin{tikzpicture}
\draw(0,0) circle(2pt) node[left=0.2em,scale=0.8]{$1$};
\draw(0,1) circle(2pt) node[left=0.2em,scale=0.8]{$2$};
\draw(0,2) circle(2pt) node[left=0.2em,scale=0.8]{$3$};
\draw(0,3) circle(2pt) node[left=0.2em,scale=0.8]{$4$};
\fill(1,0) circle(2pt) node[below]{$\mu_1$};
\draw(2,0) circle(2pt);
\draw[dashed,shorten >=2pt,shorten <=2pt] (0,0) -- (1,0);
\draw[dashed,shorten >=2pt,shorten <=2pt] (1,0) -- (2,0);
\qarrow{0,1}{1,0} node[midway,below]{$a_2$};
\qarrow{1,0}{0,2} node[pos=0.7,below]{$a_3$};
\qarrow{0,3}{1,0} node[pos=0.5,right]{$a_4$};
\draw(3,0) circle(2pt);
\draw(3,1) circle(2pt);
\draw(3,2) circle(2pt);
\draw(3,3) circle(2pt);
\draw[gray!60,very thick,shorten >=2pt,shorten <=2pt] (2,0) to[out=0,in=180] (3,3);
\draw[gray!60,very thick,shorten >=2pt,shorten <=2pt] (0,1) --++(2,0) to[out=0,in=180] (3,0);
\draw[gray!60,very thick,shorten >=2pt,shorten <=2pt] (0,2) --++(2,0) to[out=0,in=180] (3,1);
\draw[gray!60,very thick,shorten >=2pt,shorten <=2pt] (0,3) --++(2,0) to[out=0,in=180] (3,2);
\foreach \i in {1,2,3} \filldraw[fill=white] (2,\i) circle(2pt);
\node at (1.5,-1) {$\widetilde{\cN}_\phi$};
\begin{scope}[xshift=6cm,yshift=1.5cm]
\draw(0,0) circle(2pt);
\fill(1.5,0) circle(2pt) node[right=0.2em]{$\mu_1$};
\draw[dashed,shorten >=2pt,shorten <=2pt] (0,0) to[bend left=20] (1.5,0);
\draw[dashed,shorten >=2pt,shorten <=2pt] (0,0) to[bend right=20] (1.5,0);
\draw[qarrow,shorten >=2pt,shorten <=2pt] (0,0) to[bend left=55] node[midway,above,scale=0.8]{$\sum_{j=2}^N[-a_j]_+$} (1.5,0);
\draw[qarrow,shorten >=2pt,shorten <=2pt] (1.5,0) to[bend left=55] node[midway,below,scale=0.8]{$\sum_{j=2}^N[a_j]_+$} (0,0);
\node at (0.75,-2.5) {$\cN_\phi$};
\end{scope}
\end{tikzpicture}
\end{align*}
The graded adjacency and Neumann--Zagier matrices are size one, and computed as (cf. \cite[Theorem 4.1]{Mizuno:thesis})
\begin{align*}
    &N^0(t)=1+t^N, \quad N^+(t)=\sum_{j=2}^N [a_j]_+t^{j-1}, \quad N^-(t)=\sum_{j=2}^N [-a_j]_+t^{j-1},\\
    &A^+(t)=1-\sum_{j=2}^N[a_j]_+t^{j-1}+t^N,\quad A^-=1-\sum_{j=2}^N[-a_j]_+t^{j-1} + t^N.
\end{align*}
From their specialization to $t=1$, we get
\begin{align*}
    R_\phi \cong \cO(\mathrm{Glue}_\phi)&= \bQ[z_+^{\pm 1},z_-^{\pm 1}]/(z_+ + z_- -1, z_+^{-A^+}z_-^{A^-}-1) \\
    &\cong \bQ[z^{\pm 1}][(z-1)^{-1}]/(z^{-A^+}(1-z)^{A^-}-1).
\end{align*}
From \cref{cor:T_NZ_data}, we get
\begin{align}\label{eq:Kronecker_T}
    T(t)&= A^+(t) z_- + A^-(t) z_+ = A^+(t) + (A^-(t) -A^+(t))z.
\end{align}
Here, we used the symmetry $t^N A_\pm(t^{-1})=A_\pm(t)$ that follows from the palindromicity of $\mathbf{a}$. 

\begin{lem}\label{lem:FM_poly_presentation}
Assume $\mathbf{a}\neq \mathbf{0}$ and $(A^+,A^-) \neq (0,0)$. 
We have a polynomial presentation
\begin{align*}
    R_\phi \cong \bQ[z]/(f_\phi),
\end{align*}
where
\begin{align}\label{eq:FM_poly_presentation}
    f_\phi:=\begin{cases}
        z^{A^+}-(1-z)^{A^-} & \mbox{if $A^+,A^- > 0$}, \\
        z^{-A^+}-(1-z)^{-A^-} & \mbox{if $A^+,A^- < 0$}, \\
        z^{A^+}(1-z)^{-A^-}-1 & \mbox{if $A^+ > 0$ and $A^-< 0$}, \\
        z^{-A^+}(1-z)^{A^-}-1 & \mbox{if $A^+< 0$ and $A^-> 0$}, \\
        \sum_{j=0}^{|c|-1} z^j & \mbox{if $(A^+,A^-)=(c,0)$ with $c \neq 0$}, \\
        \sum_{j=0}^{|c|-1} (1-z)^j & \mbox{if $(A^+,A^-)=(0,c)$ with $c \neq 0$}. 
    \end{cases}
\end{align}
\end{lem}
\paragraph{\textbf{The case $\mathbf{a}=\mathbf{0}$:}} From \eqref{eq:FM_fixed_y}, the fixed point equation becomes $y^2=1$, which have solutions $y=\pm 1$. We have $\Fix_\phi^\ast(\overline{\bQ})=\{1\} \subsetneq \Fix_\phi(\overline{\bQ})$, where the solution $y=1$ corresponds to $z=1/2$. From \eqref{eq:FM_y_action}, we can directly compute the characteristic polynomial of the Jacobian, and hence obtain $T(t)=1+t^N$ at $y=\pm 1$. In particular, $\phi$ is non-degenerate and 
\begin{align*}
    \Ind_{g,t}^\phi=2(1+t^N)^{g-1}
\end{align*}
for all $g \geq 0$. This shows $\Ind_{g,t}^\phi \in \bZ[t]$ for $g \geq 1$. We also have $\Ind_0^\phi=1$, and hence $\Ind_g^\phi \in \bZ$ for all genera. 

\smallskip 

\paragraph{\textbf{The case $(A^+,A^-)=(0,0)$:}} The gluing variety is $1$-dimensional, and hence the indices are undefined.

\begin{rem}\label{rem:A_parity}
By the palindromicity of $\mathbf{a}$, it cannot happen that both $A^+$ and $A^-$ are odd. Indeed, if the length $N-1$ is even, then each number appears twice, and hence $A^+,A^-$ are even. Otherwise, the central number can only contribute to one of them, and hence at most one of them can be odd.
\end{rem}

\subsubsection{Integrality of the refined index for $g \geq 1$}

\begin{thm}\label{thm:Fordy_Marsh}
For any mutation loop $\phi$ with period $1$ associated with 
$(A^+,A^-) \neq (0,0)$, we have
$\Ind_g^\phi \in \bZ$ for all $g \geq 1$. If moreover $A^+ \neq A^-$, then $\Ind_{g,t}^\phi \in \bZ[t]$ for $g \geq 1$.
\end{thm}

\begin{proof}
The case $\mathbf{a}= \mathbf{0}$ has been proved. Below we assume $\mathbf{a} \neq \mathbf{0}$ so that $\Fix_\phi =\Fix_\phi^\ast $ holds.

\smallskip 

\paragraph{\textbf{The case $A^+=A^- \neq 0$:}} We directly obtain $T=A^+ \in \bZ$ from \eqref{eq:Kronecker_T}, which shows $\Ind_g^\phi \in \bZ$ for all $g \geq 1$ by \cref{lem:index_integral}. See \cref{ex:refined_non_integral} for a non-integral example of refined indices. 

\smallskip 

\paragraph{\textbf{The case $A^+ \neq A^-$:}}
We verify the conditions in \cref{prop:shape_integral}.
To see the integrality of shape parameters,
note that $\pm f_\phi$ in \eqref{eq:FM_poly_presentation} provides a monic polynomial for $z$ with integral coefficients, where the highest terms do not cancel since $A^+ \neq A^-$. Therefore $z \in \cO_{R_\phi}$. 


To see the non-degeneracy, suppose first that $A^+A^-=0$. 
Then $(A^+,A^-)=(0,c)$ or $(c,0)$ with $c \neq 0$. 
The solutions to the gluing equation are explicitly given by $(z^+,z^-)=(1-z,z)$ and $(z,1-z)$, respectively, where $z^{|c|}=1$ and $z \neq 1$. 
We have $T_z=c(1-z) \neq 0$ at these points, and hence $\phi$ is non-degenerate. 

In the case $A^+A^- \neq 0$, 
we prove that the equations 
\begin{align*}
    z^{-A^+}(1-z)^{A^-} = 1, \quad A^+ + (A^- -A^+)z=0
\end{align*}
do not have a common solution. 
Solving the second equation, we get $z=A^+/(A^+-A^-)$ and $1-z=-A^-/(A^+-A^-)$. Then the first equation becomes
\begin{align*}
    \bigg(\frac{A^+}{A^+-A^-} \bigg)^{A^+} = \bigg(\frac{-A^-}{A^+-A^-} \bigg)^{A^-}.
\end{align*}
This equation has no solution by applying the elementary \cref{lem:arithmetic} below to $a:=A^+$, $b:=A^-$. Hence $\phi$ is non-degenerate. 

Therefore \cref{prop:shape_integral} proves $\Ind_{g,t}^\phi \in \bZ[t]$ for $g \geq 1$ when $A^+ \neq A^-$.


\end{proof}

\begin{lem}\label{lem:arithmetic}
There are no integers $a,b$ satisfying $ab(a-b)\neq 0$ and $a^a(a-b)^b=\pm b^b(a-b)^a$.
\end{lem}

\begin{proof}
Let $d:=\gcd(a,b) \in \bZ_{>0}$ and $x:=a/d$, $y:=b/d$. Then the equation is equivalent to $x^a(x-y)^b=\pm y^b(x-y)^a$. Since $\gcd(x,y)=\gcd(x,x-y)=\gcd(y,x-y)=1$, the three integers $x, y$, and $x-y$ are pairwise coprime. After moving all negative exponents to the opposite side, we obtain an equality of two monomials in the pairwise coprime integers $x, y$, and $x-y$ with non-negative exponents.
Then by the uniqueness of the prime factorization, we get $x,y,x-y\in \{\pm 1\}$, which is not possible. 
\end{proof}

\begin{ex}\label{ex:refined_non_integral}
When $A^+=A^- \neq 0$, we have the following example with $\Ind_{g,t}^\phi \notin \bZ[t]$. 
Let $N=6$ and $\mathbf{a}=(1,1,-4,1,1)$. Then $A^+=A^-=-2$. The gluing equation becomes $z^2(1-z)^{-2}=1$, which reduces to $2z-1=0$. We have
\begin{align*}
    A^+(t)&=1-t-t^2-t^4-t^5+t^6,\\
    A^-(t)&=1-4t^3+t^6,
\end{align*}
and hence $T(t)=A^+(t)+\frac{1}{2}(t+t^2-4t^3+t^4+t^5) \notin \bZ[t]$, where we substituted $z=1/2$.
\end{ex}

\subsubsection{Index for $g=0$}

\begin{thm}\label{thm:FM_ind0}
Let $\phi$ be the period $1$ mutation loop associated with $\mathbf{a}\neq \mathbf{0}$ and $(A^+,A^-)\neq (0,0)$. 
\begin{enumerate}
    \item If $A^+A^-<0$, then $\Ind_0^\phi=0$. 
    \item If $A^+A^->0$ and $A^+ \neq A^-$, then $\Ind_0^\phi=\sgn A^+ \in \{\pm 1\}$.
    \item If $A^+=A^-=c \neq 0$, then $\Ind_0^\phi=\displaystyle\frac{|c|-1}{c}$. 
    \item If $(A^+,A^-)=(c,0)$ or $(0,c)$ with $c \neq 0$, then $\Ind_0^\phi=\displaystyle \frac{|c|-1}{2c}$.
\end{enumerate}
\end{thm}

\begin{proof}
In the presentation $R_\phi^\bC=\bC[z]/(f_\phi)$ given in \cref{lem:FM_poly_presentation}, $\{f_\phi\}$ gives a Gr\"obner basis of the defining ideal with respect to $d=(1)$, namely the ordinary degree. We have
\begin{align*}
    \deg f_\phi=\begin{cases}
        \max\{|A^+|,|A^-|\} & \mbox{if $A^+A^-> 0$ and $A^+\neq A^-$}, \\
        |A^+|+|A^-| & \mbox{if $A^+A^- < 0$}, \\
        |c|-1 & \mbox{if $(A^+,A^-)=(c,c)$, $(c,0)$ or $(0,c)$ with $c \neq 0$},
    \end{cases}
\end{align*}
and $N:=\deg f_\phi-1$. Let $J:=\frac{d}{dz}f_\phi $ denote the Jacobian. We are going to find $U \in \bC[z]$ satisfying 
\begin{align*}
    J\equiv UT \mod (f_\phi), \quad \deg U \leq N
\end{align*}
in each case. Then by \cref{cor:SiebertTian}, the genus-$0$ index is computed as $\Ind_0^\phi=\mathrm{Res}_{f_\phi}(U)$.
\begin{description}
\item[The case $A^+A^-<0$]
First consider the case $A^+>0$ and $A^-<0$.
Then we have $U=z^{A^+-1}(1-z)^{-A^--1}$. Indeed, by definition
\begin{align*}
    J &=A^+ z^{A^+-1}(1-z)^{-A^-} +A^- z^{A^+}(1-z)^{-A^--1} \\
    &= z^{A^+-1}(1-z)^{-A^--1}(A^+ (1-z) + A^- z) = z^{A^+-1}(1-z)^{-A^--1} T,
\end{align*}
and $\deg U =|A^+|+|A^-|-2=N-1$. Therefore, by \cref{cor:SiebertTian}, we obtain $\Ind_0^\phi=0$. The other case $A^+<0$ and $A^->0$ is treated similarly by replacing the formulas by $(A^+,A^-) \mapsto (-A^+,-A^-)$. 

\item[The case $A^+A^->0$ and $A^+ \neq A^-$] First consider the case $A^+,A^->0$. Then we have $U= z^{A^+-1}+ (1-z)^{A^--1}$. Indeed,
\begin{align*}
    J=A^+ z^{A^+-1} +A^- (1-z)^{A^--1} 
\end{align*}
by definition, and we find that $UT-J =(A^- - A^+)f_\phi$ by a direct calculation. Therefore, $J \equiv UT$ mod $(f_\phi)$, and
\begin{align*}
    \deg U =\max\{A^+,A^-\}-1 =N.
\end{align*}
Observe that $J=\max\{A^+,A^-\} U=(N+1)U$ modulo lower degree terms. For example, if $A^+>A^-$, then
\begin{align*}
    J=A^+ U + (A^--A^+)(1-z)^{A^- -1}.
\end{align*}
The other cases are treated similarly. Therefore, by the normalization $\mathrm{Res}_{f_\phi}(J)=\dim_\bC R=N+1$ in \cref{prop:SiebertTian}, we obtain
\begin{align*}
    \Ind_0^\phi = \mathrm{Res}_{f_\phi}(U) =\frac{1}{N+1}\mathrm{Res}_{f_\phi}(J) = 1.
\end{align*}
The case $A^+,A^-<0$ is treated similarly by replacing the formulas by $(A^+,A^-) \mapsto (-A^+,-A^-)$. 
    \item[The case $A^+=A^-=c$] In this case, $T=c$ is a constant. By \cref{rem:A_parity}, $|c|$ is even. 
    Hence,
    \begin{align*}
        \Ind_0^\phi=\mathrm{Res}_{f_\phi}\bigg(\frac{J}{c}\bigg) = \frac{\dim_\bC R_\phi^\bC}{c}.
    \end{align*}
    The gluing equation becomes $(z^{-1}-1)^{|c|}=1$. This is equivalent to $w^{|c|}=1$ and $z=(1+w)^{-1}$. 
    Since $|c|$ is even, all the $|c|$-th roots of unity but $-1$ corresponds to the fixed points. Therefore, $\dim_\bC R^\bC=|c|-1$.
    \item[The case $A^-=0$] Let $c:=A^+ \neq 0$. The gluing equation becomes $z^{|c|}=1$, $z \neq 1$, and we have $T=c(1-z)$. Therefore from the definition, we obtain
    \begin{align*}
        \Ind_0^\phi= \frac{1}{c}\sum_{z^c=1,~z \neq 1} \frac{1}{1-z}.
    \end{align*}
    Observe that 
    \begin{align*}
        \frac{1}{1-z} + \frac{1}{1-z^{-1}} = 1
    \end{align*}
    for any $z \neq 1$.
    If $|c|$ is odd, then the solutions of $z^c=1$, $z \neq 1$ are grouped into $(|c|-1)/2$ conjugate pairs, each of which contributes by $1$. Therefore 
    \begin{align*}
        \Ind_0^\phi= \frac{|c|-1}{2c}.
    \end{align*}
    If $|c|$ is even, then $z=-1$ is a self-conjugate solution, which contributes by $1/(1-z)=1/2$. The other solutions are grouped into $(|c|-2)/2$ conjugate pairs. Therefore, 
    \begin{align*}
        \Ind_0^\phi= \frac{1}{c}\bigg(\frac{|c|-2}{2} + \frac{1}{2}\bigg) = \frac{|c|-1}{2c}.
    \end{align*}
\end{description}
\end{proof}


\cref{tab:index-g0} is obtained by numerical computation using Mathematica. 
Observe that non-integer values appear on the loci $A^+=0$, $A^-=0$ or $A^+=A^-$.

\newcommand{\notrealized}{$\bullet$}

\begin{table}[ht]
  \centering
  \renewcommand{\arraystretch}{1.15}
  \setlength{\tabcolsep}{5pt}
  \begin{tabular}{c|cccccccc}
    \hline
    \backslashbox{\(A^+\)}{\(A^-\)}
      & \(-5\) & \(-4\) & \(-3\) & \(-2\)
      & \(-1\) & \(0\) & \(1\) & \(2\) \\
    \hline
    \(-5\)
      & \notrealized & \(-1\) & \notrealized & \(-1\)
      & \notrealized & \(-\frac{2}{5}\) & \notrealized & \(0\) \\
    \(-4\)
      & \(-1\) & \(-\frac{3}{4}\) & \(-1\) & \(-1\)
      & \(-1\) & \(-\frac{3}{8}\) & \(0\) & \(0\) \\
    \(-3\)
      & \notrealized & \(-1\) & \notrealized & \(-1\)
      & \notrealized & \(-\frac{1}{3}\) & \notrealized & \(0\) \\
    \(-2\)
      & \(-1\) & \(-1\) & \(-1\) & \(-\frac{1}{2}\)
      & \(-1\) & \(-\frac{1}{4}\) & \(0\) & \(0\) \\
    \(-1\)
      & \notrealized & \(-1\) & \notrealized & \(-1\)
      & \notrealized & \(0\) & \notrealized & \(0\) \\
    \(0\)
      & \(-\frac{2}{5}\) & \(-\frac{3}{8}\)
      & \(-\frac{1}{3}\) & \(-\frac{1}{4}\)
      & \(0\) & N/A & \(0\) & \(\frac{1}{4}\) \\
    \(1\)
      & \notrealized & \(0\) & \notrealized & \(0\)
      & \notrealized & \(0\) & \notrealized & \(1\) \\
    \(2\)
      & \(0\) & \(0\) & \(0\) & \(0\)
      & \(0\) & \(\frac{1}{4}\) & \(1\)
      & \textcolor{red}{\(1\)} \\[2mm]
  \end{tabular}
  \caption{Possible values of \(\Ind_0^\phi\) for
    \(A^+,A^-\in\{-5,\dots,2\}\).
    Bullet entries are excluded by \cref{rem:A_parity}.
    N/A means that the index is undefined.
    The entry in \textcolor{red}{red} corresponds to the exceptional
    case \(\mathbf{a}=\mathbf{0}\).}
  \label{tab:index-g0}
\end{table}

    

    

\subsection{Special case: the Kronecker quivers}
Let us investigate the case $N=2$, $a_2=-K$ for $K \geq 1$ in more detail. 
The exchange matrix is $B=\begin{psmallmatrix}0 & -K \\ K & 0\end{psmallmatrix}$, which corresponds to the Kronecker quiver with $K$ arrows. We have $A^+(t)= 1+t^2$, $A^-(t)=1-Kt+t^2$. In particular, 
$A^+=2$ and $A^-=2-K$.

The values of $\Ind_g^\phi$ (\cref{tab:Kronecker_index}) and $\Ind_{g,t}^\phi$ (\cref{tab:Kronecker_refined_index}) for small values $K \geq 1$ are obtained by numerical computation using Mathematica.


\begin{table}[ht]
    \centering
    \begin{tabular}{|c|cccccc}
    \hline
    \backslashbox{$g$}{$K$} & 1 & 2 & 3 & 4 & 5 & 6 \\
    \hline
    0 & 1 & $\frac14$ & 0 & 0 & 0 & 0 \\
    1 & 2 & 1 & 3 & 4 & 5 & 6 \\
    2 & 5 & 4 & 3 & 0 & -5 & -12 \\
    3 & 15 & 16 & 9 & 16 & 35 & 72 \\
    4 & 50 & 64 & 96 & 0 & -65 & -240 \\
    5 & 175 & 256 & 357 & 1088 & 275 & 1056 \\
    6 & 625 & 1024 & 1278 & 0 & 14960 & -4032 \\
    7 & 2250 & 4096 & 6042 & 12544 & -91435 & 296448 \\
    8 & 8125 & 16384 & 26337 & 0 & 649945 & -3984384 \\[2mm]
    \end{tabular}
    \caption{Values of $\Ind_g^\phi$ for the Kronecker quiver with $K \geq 1$.}
    \label{tab:Kronecker_index}
\end{table}

\newcommand{\tabmath}[1]{{\footnotesize$\displaystyle #1$}}


\begin{table}[ht]
    \centering
    \begin{tabular}{|c|ccc}
    \hline
    \backslashbox{$g$}{$K$} & 1 & 2 & 3 \\
    \hline
    0 &
    \tabmath{\frac{2 + t + 2 t^2}{1 + t + t^2 + t^3 + t^4}} &
    \tabmath{\frac{1}{(1 + t)^2}} &
    \tabmath{\frac{3 (-1 + t)^2 (1 + t^2)}
    {1 - 3 t + 3 t^2 + 21 t^3 + 3 t^4 - 3 t^5 + t^6}}
    \\
    1 &
    \tabmath{2} &
    \tabmath{1} &
    \tabmath{3}
    \\
    2 &
    \tabmath{2 + t + 2 t^2} &
    \tabmath{(1 + t)^2} &
    \tabmath{3 (1 - t + t^2)}
    \\
    3 &
    \tabmath{2 + 2 t + 7 t^2 + 2 t^3 + 2 t^4} &
    \tabmath{(1 + t)^4} &
    \tabmath{3 (1 - 2 t + 5 t^2 - 2 t^3 + t^4)}
    \\
    4 &
    \tabmath{(2 + t + 2 t^2) (1 + t + 6 t^2 + t^3 + t^4)} &
    \tabmath{(1 + t)^6} &
    \tabmath{3 (1 - 3 t + 12 t^2 + 12 t^3 + 12 t^4 - 3 t^5 + t^6)}
    \\[2mm]
    \end{tabular}
    \caption{Values of $\Ind_{g,t}^\phi$ for the Kronecker quiver with $K \geq 1$.}
    \label{tab:Kronecker_refined_index}
\end{table}

\begin{thm}\label{thm:generating}
For $K \geq 3$, the generating function of indices is computed as
\begin{align*}
    E_\phi(s) &= -(-sK)^K+(1+2s)^2 (1+(2-K)s)^{K-2} , \\
    Z_\phi(s) &= \frac{K^{K+1}(-s)^{K-1}+K(1+2s)(1+(2-K)s)^{K-3}
\bigl((4-K)-2(K-2)s\bigr)}
{-(-sK)^K+(1+2s)^2(1+(2-K)s)^{K-2}
}.
\end{align*}
\end{thm}

\begin{proof}
Let $f(y):=(1+y)^K -y^2=0$ be the fixed point equation from \eqref{eq:FM_fixed_y}. For each root $y$ of $f$, we obtain
\begin{align*}
    b_y := T(y)=2-K \frac{y}{1+y} = \frac{2+(2-K)y}{1+y}
\end{align*}
from \eqref{eq:Kronecker_T}. By solving this relation for $y$, we obtain
\begin{align*}
    y= \frac{-b_y+2}{b_y+K-2}.
\end{align*}
From this, we obtain $f(\frac{-s+2}{s+K-2})|_{s=b_y} = f(y) =0$.
The combination
\begin{align*}
    -(s+K-2)^K f\left(\frac{-s+2}{s+K-2}\right) = -K^K + (-s+2)^2(s+K-2)^{K-2}
\end{align*}
is a monic polynomial of $s$ of degree $K$, having $K$ roots $s=b_y$. Therefore it coincides with $\widetilde{E}_\phi(s):=\prod_{y} (s -b_y)$. Then
\begin{align*}
    E_\phi(s)&=\prod_y (1+b_y s) \\
    &= (-s)^K \prod_y (-s^{-1}-b_y) \\
    &= (-s)^K \widetilde{E}_\phi(-s^{-1}) \\
    &= (-s)^K [-K^K + (s^{-1}+2)^2(-s^{-1}+K-2)^{K-2}] \\
    &= -(-sK)^K +(1+2s)^2 (1+(2-K)s)^{K-2}, 
\end{align*}
as asserted. The second equation is directly obtained by differentiation.
\end{proof}


\end{document}